\documentclass[american]{scrartcl}
\KOMAoptions{DIV=13, paper=a4, abstract=on, captions=tableheading}
\usepackage[utf8]{inputenc}
\usepackage[T1]{fontenc}
\usepackage{babel, csquotes, lmodern, microtype}
\usepackage{amsmath, amssymb, amsthm, bm, booktabs, pgfplots, siunitx, subcaption}
\usepackage{afterpage}

\usepackage{hyperref}
\hypersetup{%
   pdfauthor={W. Huang, L. Kamenski, J. Lang},
   pdftitle={Stability of explicit one-step methods
   for P1-finite element approximation of linear diffusion equations
   on anisotropic meshes},
   pdfsubject={Preprint},
}

\usepackage[capitalize]{cleveref}
\crefformat{section}{#2Sect.~#1#3}
\Crefformat{section}{#2Section~#1#3}
\crefmultiformat{section}{Sects.~#2#1#3}%
   { and~#2#1#3}{, #2#1#3}{ and~#2#1#3}
\crefformat{equation}{(#2#1#3)}
\Crefformat{equation}{Equation~(#2#1#3)}
\crefmultiformat{equation}{(#2#1#3)}%
   { and~(#2#1#3)}{, (#2#1#3)}{ and~(#2#1#3)}
\crefrangeformat{equation}{(#3#1#4) to~(#5#2#6)}
\crefformat{lemma1}{Lemma~#2#1#3}
\Crefformat{lemma1}{Lemma~#2#1#3}
\crefmultiformat{lemma1}{Lemmas~#2#1#3}%
   { and~#2#1#3}{, #2#1#3}{ and~#2#1#3}
\crefrangeformat{lemma1}{Lemmas~#3#1#4 to~#5#2#6}
\crefformat{corollary1}{Corollary~#2#1#3}
\Crefformat{corollary1}{Corollary~#2#1#3}
\crefmultiformat{corollary1}{Corollaries~#2#1#3}%
   { and~#2#1#3}{, #2#1#3}{ and~#2#1#3}
\crefrangeformat{corollary1}{Corollaries~#3#1#4 to~#5#2#6}

\makeatletter\g@addto@macro\@floatboxreset\centering\makeatother
\usetikzlibrary{external}
\usetikzlibrary{calc,shadings}
\pgfplotsset{%
   compat=newest,
   grid=major, ticks=major,
   every axis/.append style={font=\footnotesize},
}

\newcommand\D{\mathbb{D}}
\newcommand\M{\mathbb{M}}
\newcommand\cO{\mathcal{O}}
\newcommand\R{\mathbb{R}}
\newcommand\Th{\mathcal{T}_h}

\DeclareMathOperator{\Diag}{diag}
\providecommand{\diag}{\Diag}

\newcommand{\be}{\boldsymbol{e}}

\newcommand{\bS}{\boldsymbol{S}}
\newcommand{\bZ}{\boldsymbol{Z}}

\newcommand{\bU}{\boldsymbol{U}}
\newcommand{\bv}{\boldsymbol{v}}

\newcommand{\bx}{\boldsymbol{x}}
\newcommand{\bxi}{\boldsymbol{\xi}}

\newcommand{\MAM}{M^{-\frac{1}{2}} A M^{-\frac{1}{2}}}
\newcommand{\MA}{M^{-1} A}
\newcommand{\AMA}{A^{\frac{1}{2}} M^{-1} A^{\frac{1}{2}}}
\newcommand{\FDF}{{\left(F_K'\right)}^{-1} \D_K {\left(F_K'\right)}^{-T}}
\newcommand{\FF}{{\left(F_K'\right)}^{-1} {\left(F_K'\right)}^{-T}}

\newcommand{\FMkF}{{\left(F_K'\right)}^{-1} \M_K^{-1} {\left(F_K'\right)}^{-T}}
\newcommand{\detFMkF}{\det\left(\FMkF\right)}
\newcommand{\tM}{\tilde{M}}

\newcommand{\dx}{\;d\boldsymbol{x}}

\newcommand{\Abs}[1]{{\left\lvert#1\right\rvert}}
\newcommand{\abs}[1]{{\lvert#1\rvert}}
\newcommand{\Norm}[1]{{\left\lVert#1\right\rVert}}
\newcommand{\norm}[1]{{\lVert#1\rVert}}
\newcommand{\NormE}[1]{{|||#1|||}}

\newcommand{\Qali}[1]{Q_{\text{ali},#1}}
\newcommand{\Qeq}[1]{Q_{\text{eq},#1}}
\newcommand{\Q}[1]{Q_{#1}}

\DeclareMathOperator{\QD}{\Q{\D^{-1}}}
\DeclareMathOperator{\QaliM}{\Qali{\M}}
\DeclareMathOperator{\QeqM}{\Qeq{\M}}
\DeclareMathOperator{\QM}{\Q{\M}}

\theoremstyle{plain} % default
\newtheorem{theorem}{Theorem}[section]
\newtheorem{lemma1}{Lemma}[section]
\newtheorem{corollary1}{Corollary}[theorem]
\theoremstyle{remark}
\newtheorem{example}{Example}[section]
\newtheorem{remark}{Remark}[section]

\newenvironment{keywords}%
   {\begin{trivlist}\item[]{\bfseries\sffamily Keywords:}~}%
   {\end{trivlist}}
\newenvironment{AMS}%
   {\begin{trivlist}\item[]{\bfseries\sffamily AMS 2010 MSC:}~}%
   {\end{trivlist}}

\begin{document}

%*** Title **********************************************************
\title{%
   Stability of~explicit one-step methods
   for~P1-finite element approximation
   of~linear diffusion equations
   on~anisotropic meshes%
   \thanks{Supported in part by
      the NSF (USA) under Grant No.~DMS-1115118,
      the DFG (Germany) under Grant No.~KA\,3215/2-1,
      the Darmstadt Graduate Schools of Excellence
         \emph{Computational Engineering}
         and~\emph{Energy Science and Engineering}.%
   }
}
\author{Weizhang Huang%
   \thanks{%
      Department of Mathematics, the University of Kansas,
      Lawrence, KS, USA
      (\href{mailto:whuang@ku.edu}%
        {\nolinkurl{whuang@ku.edu}}).%
   }
\and Lennard Kamenski%
   \thanks{%
      Weierstrass Institute,
      Berlin, Germany
      (\href{mailto:kamenski@wias-berlin.de}%
        {\nolinkurl{kamenski@wias-berlin.de}}).%
   }
\and Jens Lang%
   \thanks{%
      Department of Mathematics,
      %Technische Universit{\"a}t Darmstadt, Germany
      TU Darmstadt, Germany
      (\href{mailto:lang@mathematik.tu-darmstadt.de}%
        {\nolinkurl{lang@mathematik.tu-darmstadt.de}}).%
   }
}

\maketitle

%--- Abstract --------------------------------------------------------
\begin{abstract}
We study the stability of explicit one-step integration schemes for the linear finite element approximation of linear parabolic equations.
The derived bound on the largest permissible time step is tight for any mesh and any diffusion matrix within a factor of $2(d+1)$, where $d$ is the spatial dimension.
Both full mass matrix and mass lumping are considered.
The bound reveals that the stability condition is affected by two factors.
The first one depends on the number of mesh elements and corresponds to the classic bound for the Laplace operator on a uniform mesh.
The other factor reflects the effects of the interplay of the mesh geometry and the diffusion matrix.
It is shown that it is not the mesh geometry itself but the mesh geometry in relation to the diffusion matrix that is crucial to the stability of explicit methods.
When the mesh is uniform in the metric specified by the inverse of the diffusion matrix, the stability condition is comparable to the situation with the Laplace operator on a uniform mesh.
Numerical results are presented to verify the theoretical findings.


% AMS
\end{abstract}

%*** INTRODUCTION **************************************************************
\section{Introduction}
\label{sec:introduction}

Adaptive meshes are commonly used for the numerical solution of partial differential equations (PDEs) to enhance computational efficiency but there are still lacks in the mathematical understanding of the effects of the variation of element size and shape on the properties of numerical schemes for solving PDEs.

In this paper, we are concerned with the stability of explicit one-step time integration of linear finite element approximation with general nonuniform simplicial meshes for the initial-boundary value problem (IBVP)
\begin{align}
   \begin{cases}
      \begin{alignedat}{3}
         & \partial_t u = \nabla \cdot \left(\D \nabla u \right),
            \qquad && \bx \in \Omega,
            \quad  && t \in \left( 0, T \right],
            \\
         & u(\bx, t) = 0,
            \qquad && \bx \in \Gamma_D,
            \quad  && t \in \left( 0, T \right],
            \\
         & \D \nabla u(\bx, t) \cdot \bm{n} = 0,
            \qquad && \bx \in \Gamma_N,
            \quad  && t \in \left( 0, T \right],
            \\
         & u(\bx, 0) = u_0(\bx),
            \qquad && \bx \in \Omega
      \end{alignedat}
   \end{cases}
   \label{eq:IBVP}
\end{align}
where $\Omega \subset \mathbb{R}^d$ ($d \ge 1$) is an interval, a bounded polygonal or polyhedral domain, $\Gamma_D \cup \Gamma_N = \partial \Omega$, $\Gamma_D$ has a positive $(d-1)$-volume, $u_0$ is a given initial function, and $\D$ is the diffusion matrix which is assumed to be symmetric and uniformly positive definite on $\Omega$.
In this study, we also assume that $\D$ is time independent, i.e.,\ $\D = \D(\bx)$.
Problem \cref{eq:IBVP} is isotropic when $\D(\bx) = \alpha (\bx) I$ for all $\bx$ in $\Omega$, where $\alpha$ is a scalar function and $I$ is the $d$-by-$d$ identity matrix. 
Otherwise, the problem is an anisotropic diffusion problem for which we shall consider in this work.
Anisotropic diffusion arises in various areas of science and engineering, including plasma physics~\cite{GL09}, petroleum reservoir simulation~\cite{EAK01,MD06}, and image processing~\cite{KM09,Wei98}.

Assume that $u_0 \in H^1_D(\Omega) = \left\{ v \in H^1(\Omega): \text{$v=0$ on $\Gamma_D$} \right\}$.
Then, if $u$ is sufficiently smooth, we have the stability estimates
\begin{align}
   \begin{cases}
   \begin{alignedat}{2}
      \Norm{u(\cdot,t)}_{L^2(\Omega)}
         &\le \Norm{u_0}_{L^2(\Omega)},
        && \qquad t \in \left( 0, T \right],
      \\[0.5ex]
      \NormE{u(\cdot,t)}_{H^1(\Omega)}
         &\le \NormE{u_0}_{H^1(\Omega)},
        && \qquad t \in \left( 0, T \right],
   \end{alignedat}
   \end{cases}
   \label{energy-est}
\end{align}
where $\NormE{u(\cdot, t)}_{H^1(\Omega)} \equiv \norm{\D^{1/2}\nabla u}_{L^2(\Omega)}$ is the energy norm of $u(\cdot, t)$.
It is essential that a numerical scheme for \cref{eq:IBVP} preserves the stability estimates.
The stability of the time integration depends on the largest eigenvalue of the system related to the numerical scheme which, in turn, depends on the underlying meshes and the coefficients of the IBVP.

For a uniform mesh and the Laplace operator, it is well known that the largest permissible time step is proportional to the square of the element diameter.

In the case of a nonuniform mesh or a variable diffusion matrix the situation becomes more complicated.
Essentially, one needs to estimate the largest eigenvalues of $\MA$, where $M$ and $A$ are the mass and stiffness matrices corresponding to the discretization of the IBVP.\@
This can be done by estimating the extreme eigenvalues of $M$ and $A$.
Tight bounds on those of the mass matrix $M$ for linear finite elements with locally quasi-uniform meshes are available in the literature and typically proportional to the extremal mesh element volumes~\cite{Fri73,GraMcL06,Wat87}, whereas those for the stiffness matrix $A$ are more difficult to obtain and only a few results are available in the literature for the case of nonuniform meshes.
For example, Fried~\cite{Fri73} shows how to obtain these bounds for the finite element approximation of the Laplace operator for general nonuniform meshes using local element mass and stiffness matrices.
A similar argument was used by Shewchuk~\cite{She02a} to develop a bound on the largest eigenvalue of $\MA$ in terms of the maximum eigenvalues of local element matrices for the case of a lumped mass matrix.
Graham and McLean~\cite{GraMcL06} study the finite/boundary element approximation of a general differential/integral operator on locally quasi-uniform meshes in terms of patch volumes and aspect ratios.
Du, Wang, and Zhu~\cite{DuWanZhu09} obtain bounds on the extreme eigenvalues of the stiffness matrix for the Galerkin approximation of a general diffusion operator  in terms of element geometry.
Zhu and Du~\cite{ZhuDu11,ZhuDu14} further develop bounds on the largest permissible
time step for time dependent problems.
It is worth mentioning that these existing works allow anisotropic meshes.
However, the interplay between the mesh geometry and the diffusion matrix is not really taken into account, which, as we will see, is crucially important for the stability of explicit integration schemes.
A notable exception is the bound obtained by Shewchuk~\cite{She02a}, which takes the effects of the diffusion coefficients fully into account; see \cref{rem:zhudu} for details and \cref{ex:aniso} for a numerical example.
Moreover, the existing analysis either employs some mesh regularity assumptions such as the local uniformity or involves parameters in final estimates that are related to mesh regularity, such as the maximum ratio of volumes of neighboring elements and/or the maximum number of elements in a patch.

The objective of this work is to provide estimates for the largest permissible time step which are accurate and tight for \emph{any mesh} and \emph{any diffusion matrix}.
We utilize bounds recently obtained by Kamenski~et~al.~\cite{KamHuaXu13} on the extreme eigenvalues of $M$ and the largest eigenvalue of $A$ for a general diffusion operator with arbitrary meshes.
The obtained stability condition expressed in terms of matrix entries is tight within a constant factor which is independent of the mesh and the diffusion matrix.
No assumption on the mesh regularity is made in the development.
We show that the alignment of the mesh with the diffusion matrix plays a crucial role in the stability condition: the largest permissible time step depends only on the number of mesh elements and the mesh geometry in relation to the diffusion matrix.
In particular, if the mesh is uniform in the metric specified by $\D^{-1}$, the stability condition is essentially the same as that for the Laplace operator with a uniform mesh.
Although we consider only linear finite elements, the presented analysis is applicable to high order finite elements without major modifications~\cite{HuaKamLan15}.

The paper is organized as follows.
We start in \cref{sec:setting} with the problem setting and a detailed description of mesh quality measures which are needed for the geometric interpretations of stability estimates.
The main results on stability are given in \cref{sec:explicit}; both the full mass matrix and mass lumping are considered.
Numerical examples to demonstrate the theoretical findings are presented in \cref{sec:numerical:examples}, including a two-dimensional groundwater flow problem.
Conclusions are drawn in \cref{sec:conclusion}.

%*** Setting **********************************************************
\section{Linear finite element approximation}
\label{sec:setting}

We consider the standard linear finite element method for the spatial discretization of IBVP~\cref{eq:IBVP}.

We assume that a family $\{ \Th\}$ of simplicial meshes is given for $\Omega$.
While having adaptive meshes in mind, we consider the meshes to be general nonuniform ones, which may contain elements of small size and large aspect ratio.
Let $K$ be an arbitrary element of $\Th$, $\hat{K}$ the \emph{reference element}, and $\omega_i$ the element \emph{patch} of the $i^{\text{th}}$ vertex (\cref{fig:mesh:elements}).
Element and patch volumes are denoted by
\[
   \Abs{K}
   \quad \text{and} \quad
   \Abs{\omega_i} = \sum_{K \in \omega_i} \Abs{K}.
\]
For each mesh element $K\in \Th$ let  $F_K$ be the invertible affine mapping from $\hat{K}$ to $K$ (\cref{fig:mesh:elements}) and  $F_K'$ its Jacobian matrix.
Note that $F_K'$ is a constant matrix with $\det(F_K') = \Abs{K}$ (for simplicity, we assume that $\hat{K}$ is equilateral with $\abs{\hat{K}} = 1$).

%--- notation figure ---
\begin{figure}[t]
   \tikzsetnextfilename{hkl2013-notation}
   %--- notation figure ---
\begin{tikzpicture}[scale = 0.9]
   % patch nodes
   \path ( 0.0,  0.0) coordinate (N0); 
   \path ( 2.0,  0.0) coordinate (N1);  
   \path ( 0.4,  2.0) coordinate (N2);
   \path (-1.8,  1.2) coordinate (N3);
   \path (-1.4, -1.0) coordinate (N4);
   \path ( 0.4, -1.6) coordinate (N5);
   % triangles
   \draw [] (N0) -- (N1) -- (N2) -- cycle;
   \draw [] (N0) -- (N2) -- (N3) -- cycle;
   \draw [fill = gray!15] (N0) -- (N3) -- (N4) -- cycle;
   \draw [] (N0) -- (N4) -- (N5) -- cycle;
   \draw [] (N0) -- (N5) -- (N1) -- cycle;
   % reference element
   \draw [fill = gray!15] (-5.0, -1.0) -- (-4.0, 0.73) -- (-3.0, -1.0) -- cycle;
   % edges and labels
   \path (-4.0,  0.1) coordinate (Khat);
   \path (-1.1,  0.3) coordinate (K);
   \node [below] at (Khat)  {$\hat{K}$};
   \node [below] at (K)  {$K$};
   \node [above right]  at (N0)  {node $i$};
   \node [right] at (1.4, -1.2)  {patch $\omega_i$};
   \filldraw [black] (N0)  circle (3.4pt);
   % arrows
   \path [->, line width = 1pt] ($(Khat) + (0.10, 0.05)$) edge [bend left] 
      node [above left] {$F_K$} ($(K) + (-0.15, 0.05)$);
   \path [->, line width = 1pt] ($(K) + (-0.15, -0.6)$) edge [bend left] 
      node [below right] {$F_K^{-1}$} ($(Khat) + (0.15, -0.75)$);
\end{tikzpicture}    
   \caption{Reference and mesh elements, mapping $F_K$,
      $i^{\text{th}}$ node and its patch $\omega_i$\label{fig:mesh:elements}}
\end{figure}
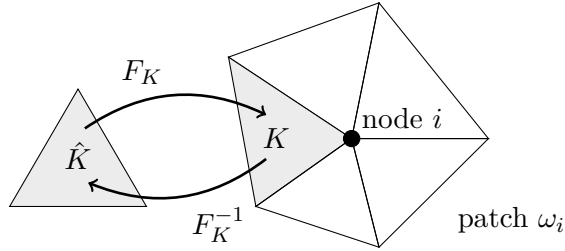

Let $V^h$ be the linear finite element space associated with mesh $\Th$.
Defining $V^h_D=V^h\cap H^1_D(\Omega)=\left\{ v^h \in V^h: \text{$v^h=0$ on $\Gamma_D$} \right\}$, the piecewise linear finite element solution $u^h(t) \in V^h_D$, $t \in \left( 0, T \right]$ is defined by
\begin{align}
   \int_\Omega v^h \partial_t u^h \dx
      &= - \int_\Omega \nabla v^h \cdot \D \nabla u^h \dx,
      && \forall v^h \in V^h_D, \quad t \in \left( 0, T \right],
   \label{eq:FEM:i}
   \\
\intertext{subject to the initial condition}
   \int_\Omega u^h(\bx, 0) v^h \dx
      &= \int_\Omega u_0(\bx) v^h \dx,
      && \forall v^h \in V^h_D .
   \label{eq:FEM:ii}
\end{align}
We denote the number of the elements of $\Th$ by $N$ and the number of the interior vertices plus the vertices associated with the Neumann boundary condition by $N_{vi}$.
If we express $u^h$ as
\[
   u^h(\bx,t) = \sum_{j=1}^{N_{vi}} u^h_j(t) \phi_j (\bx),
\]
where $\phi_j$ is the linear basis function associated the $j^{\text{th}}$ vertex ($j = 1, \dotsc, N_{vi}$), from \cref{eq:FEM:i,eq:FEM:ii} we obtain
\begin{equation}
   M \bU_t = - A \bU,
   \qquad \bU(0) = \bU_0,
   \label{eq:fem:system}
\end{equation}
where $\bU = {\left(u^h_1,\dotsc, u^h_{N_{vi}}\right)}^T$ and $M$ and $A$ are the mass and the stiffness matrices,
\begin{equation}
   M_{ij} = \int_\Omega \phi_i \phi_j \dx,
   \qquad
   A_{ij} = \int_\Omega \nabla \phi_i \cdot \D \nabla \phi_j \dx ,
   \qquad i, j = 1, \dotsc, N_{vi}.
   \label{eq:AM:i}
\end{equation}
We shall investigate how the geometry of the mesh and the anisotropy of the diffusion matrix affect the stability of explicit one-step methods for integrating \cref{eq:fem:system}.
In the following we assume that the mesh is fixed for all time steps.

%--- notation -------------------------------------------------------------
\subsection{Mathematical description of nonuniform meshes; mesh quality measures}
\label{sec:mesh:notation}
An adaptive mesh, which is typically nonuniform, can be generated as a uniform one in the metric specified by a given metric tensor, which is always assumed to be symmetric and uniformly positive definite in $\Omega$~\cite{Hua05a}.
On the other hand, a metric tensor can be defined for any given mesh such that all elements are uniform in the metric specified by this tensor~\cite{HuaRus11}.
Thus, it is natural to consider nonuniform meshes in relation to a given metric tensor.
In the following, we describe several quality measures and mathematical characterizations for (nonuniform) meshes in terms of a given metric tensor $\M = \M(\bx)$.
As we will see in \cref{sec:explicit}, the matching between the mesh metric tensor and the diffusion matrix plays a crucial role for the stability condition.
In our analysis, we slightly adjust the original definitions of the mesh quality measures in~\cite{Hua05} (see also~\cite{Hua07,HuaRus11}).

Let
\begin{equation}
   \M_K = \frac{1}{\Abs{K}} \int_K \M \dx,
   \qquad
   \Abs{K}_{\M} = \Abs{K} {\det(\M_K)}^{\frac{1}{2}},
   \qquad
   \Abs{\Omega}_{\M,h} = \sum_{K \in \Th} \Abs{K}_{\M} .
   \label{eq:sigmah:def}
\end{equation}
Note that $\M_K$ is the average of $\M$ over the element $K$ and $\Abs{K}_{\M}$ and $\Abs{\Omega}_{\M,h}$ are approximate volumes of $K$ and $\Omega$ in the metric $\M$, viz.,
\[
   \Abs{K}_{\M} \approx \int_K {\det\bigl(\M(\bx)\bigr)}^{\frac{1}{2}} \dx
   \quad \text{and} \quad
   \Abs{\Omega}_{\M,h}
      \approx \sum_{K \in \Th} \int_K {\det\bigl(\M(\bx)\bigr)}^{\frac{1}{2}} \dx
      = \Abs{\Omega}_{\M}.
\]
Hereafter, without confusion we will call $\Abs{K}_{\M}$ and $\Abs{\Omega}_{\M,h}$ the volumes of $K$ and $\Omega$ in the metric $\M$, respectively.
We also define the \emph{average diameter} of element $K$ and the \emph{global average element diameter} with respect to $\M$ as
\begin{align*}
   h_{K,\M} = \Abs{K}_{\M}^{\frac{1}{d}}
   \quad \text{and} \quad
   h_{\M} = {\left(\frac{1}{N}\Abs{\Omega}_{\M,h}\right)}^{\frac{1}{d}} .
\end{align*}
The diameter $h_K$ of $K$ is defined as the length of the longest edge of $K$.

%--- Equidistribution measure --------------------------------------------------
With these notations, we now are ready to describe the mesh quality measures.
The first one, the \emph{equidistribution} quality measure, is defined as the ratio of the average element volume to the volume of $K$, both measured in the metric specified by $\M_K$,
\begin{equation}
   \QeqM(K)
   = \frac{\frac{1}{N} \Abs{\Omega}_{\M,h}}{\Abs{K}_{\M}}
   = {\left( \frac{h_{\M}}{h_{K,\M}} \right)}^d .
   \label{eq:Qeq:def}
\end{equation}
It satisfies
\begin{equation}
   0 < \QeqM(K) < \infty,
   \qquad
   \frac{1}{N} \sum_{K \in \Th} \frac{1}{\QeqM(K)} = 1,
   \qquad
   \max_{K \in \Th} \QeqM(K) \ge 1.
   \label{eq:Qeq:properties}
\end{equation}

%--- Alignment measure --------------------------------------------------------
The second one, the \emph{alignment} quality measure, is local (elementwise) and measures how closely the principal directions of the circumscribed ellipsoid of $K$ are aligned with the eigenvectors of $\M_K$ and the semi-lengths of the principal axes are inversely proportional to the square root of the eigenvalues of $\M_K$.
It is defined as
\begin{equation}
   \QaliM(K)
      = \frac{\Norm{\FMkF}_2 }
             {{\detFMkF}^{\frac{1}{d}}} = h_{K,\M}^2 \Norm{\FMkF}_2.
   \label{eq:Qali:def}
\end{equation}
Since $\Norm{\FMkF}_2 \ge {\detFMkF}^{\frac{1}{d}}$, $\QaliM$ always satisfies
\[
   1 \le \QaliM(K) < \infty
\]
with $\QaliM(K) = 1$ if and only if $K$ is equilateral with respect to $\M_K$.
The alignment quality measure can be seen as an alternative to the aspect ratio of $K$ in the metric specified by $\M_K$ and it satisfies
\begin{equation}
   \QaliM(K) \le \hat{h}^2
      \cdot {\left(\frac{h_{K,\M}}{\rho_{K,\M}} \right)}^2 ,
\label{eq:Qali:def-2}
\end{equation}
where $\hat{h}$ is the length of the longest edge of $\hat{K}$ and $\rho_{K,\M}$ is the diameter of the largest sphere inscribed in the element $K$ viewed in the metric $M_K$.
To show this, we consider two points $\bx_1, \bx_2 \in K$ and the corresponding points $\bxi_1= F_K^{-1}(\bx_1)$ and $\bxi_2 = F_K^{-1}(\bx_2)$ in $\hat{K}$.
The distance between $\bx_1$ and $\bx_2$ in the metric $\M_K$ is
\begin{align*}
   \Norm{\bx_1 - \bx_2}^2_{\M_K}
      &= {\left(\bx_1 - \bx_2 \right)}^T \M_K \left(\bx_1 - \bx_2 \right)
       = {\left(\bxi_1 - \bxi_2 \right)}^T {\left(F_K'\right)}^{T}
         \M_K F_K' \left(\bxi_1 - \bxi_2 \right)
      \\
      &= \Norm{\bxi_1 - \bxi_2}_2^2
         \cdot
         \frac{{\left(\bxi_1 - \bxi_2 \right)}^T}{\Norm{\bxi_1 - \bxi_2}_2}
            {\left(F_K'\right)}^{T} \M_K F_K'
         \frac{\left(\bxi_1 - \bxi_2 \right)}{\Norm{\bxi_1 - \bxi_2}_2}
      \\
      &\le \hat{h}^2
         \cdot
         \frac{{\left(\bxi_1 - \bxi_2 \right)}^T}{\Norm{\bxi_1 - \bxi_2}_2}
            {\left(F_K'\right)}^{T} \M_K F_K'
         \frac{\left(\bxi_1 - \bxi_2 \right)}{\Norm{\bxi_1 - \bxi_2}_2}
      .
\end{align*}
If we take the minimum over all pairs of opposing points on the largest sphere inscribed in the element $K$ viewed in the metric $M_K$, then
\[
   \rho_{K,\M}^2
      \le \hat{h}^2 \lambda_{\min} \bigl( {\left(F_K'\right)}^T \M_K F_K' \bigr).
\]
Hence,
\begin{equation}
   \Norm{\FMkF}_2
      = \frac{1}{\lambda_{\min} \bigl( {\left(F_K'\right)}^T \M_K F_K' \bigr)}
      \le {\left(\frac{\hat{h}}{\rho_{K,\M}} \right)}^2 ,
\label{eq:Qali:def-3}
\end{equation}
which, together with  \cref{eq:Qali:def}, gives \cref{eq:Qali:def-2}.

%--- Overall element measure -------------------------------------------------
The \emph{element quality} measure is defined as a combination of $\QaliM$ and $\QeqM$,
\begin{equation}
   \QM(K)
      = \QaliM(K) \cdot {\bigl(\QeqM(K)\bigr)}^{\frac{2}{d}}
      = h_{\M}^2 \Norm{\FMkF}_2 .
   \label{eq:QM:def}
\end{equation}
It measures how far $K$ is from being equilateral with a constant volume when viewed in the metric specified by $\M$.
By definition and from \cref{eq:Qali:def-3} it follows that
\begin{equation}
   0  < \QM(K)
      \le \hat{h}^2 {\left( \frac{h_{\M}}{\rho_{K,\M}} \right)}^2
      < \infty
   .
   \label{eq:QM:geometric}
\end{equation}
When a mesh is uniform with respect to $\M$ (we will refer to it as an \emph{$\M$-uniform} mesh) then it satisfies
\begin{equation}
   \QaliM(K) = 1
   \quad \text{and} \quad
   \QeqM(K) = 1,
   \qquad \forall K \in \Th,
   \label{eq:eq:ali:muniform}
\end{equation}
which is equivalent to
\begin{equation}
   \QM(K) = 1,
   \qquad \forall K \in \Th.
   \label{eq:eq:ali:i}
\end{equation}
Indeed, \cref{eq:eq:ali:i} follows directly from \cref{eq:eq:ali:muniform}.
On the other hand, since $\QaliM \ge 1$, \cref{eq:eq:ali:i} implies $\QeqM(K) \le 1$ for all $K$.
Due to the property \cref{eq:Qeq:properties}, the latter is only possible if $\QeqM(K) = 1$ for all $K$, which, in turn, implies $\QaliM(K) = 1$ for all $K$.

It is worth mentioning that an $\M$-uniform mesh satisfies
\begin{equation}
   \FMkF = h_{\M}^{-2} I,
   \quad \forall K \in \Th,
   \label{eq:muniform:fmf}
\end{equation}
since \cref{eq:eq:ali:muniform} implies that all eigenvalues of $\FMkF$ are equal to $h_{\M}$.
On the other hand, when a mesh is far from being $\M$-uniform, then
\[
   \QaliM(K) \gg 1
   \quad \text{or} \quad
   \max\limits_K \QeqM(K) \gg 1
\]
and therefore
\[
   \max\limits_K \QM(K) \gg 1 .
\]

%--- auxiliary results ---
%-------------------------------------------------------------------------------
\subsection{Preliminary results}
\label{sec:auxiliary:results}

In this subsection we present a few properties of the mass matrix $M$ and the stiffness matrix $A$ of linear finite elements, which will be used repeatedly in our analysis.
Throughout the paper the less-than-or-equal-to sign between matrix terms means that the difference between the right-hand side and left-hand side terms is positive semidefinite.

%--- Lemma: mass matrix --------------------------------------------------------
\begin{lemma1}[{\cite[Sect.~3]{KamHuaXu13}}]
\label{thm:mass:matrix}
The linear finite element mass matrix $M$ and its diagonal part $M_D$ satisfy
\begin{equation}
   \frac{1}{2} M_D \le M \le \frac{d+2}{2} M_D
   \quad \text{and} \quad
   M_{ii} = \frac{2 \Abs{\omega_i}}{(d+1) (d+2)} ,
   \qquad i=1,\dotsc,N_{vi} .
   \label{eq:M}
\end{equation}
\end{lemma1}

%--- Lemma: lumped mass matrix ------------------------------------------------
\begin{lemma1}
\label{thm:mass:matrix:lumped}
Let $M_{lump}$ be the lumped linear finite element mass matrix defined through
\[
    M_{ii, lump} = \int\limits_\Omega \phi_i(\bx) \cdot \sum\limits_{j=1}^{N_{vi}} \phi_j(\bx) \dx,
    \qquad i=1,\dotsc,N_{vi}.
\]
Then
\begin{equation}
   \frac{2 \Abs{\omega_i}}{(d+1) (d+2)}
   \le M_{ii, lump}
   \le \frac{ \Abs{\omega_i}}{d+1}.
   \label{eq:ML}
\end{equation}
\end{lemma1}
\begin{proof}
Since
\[
   \phi_i(\bx) \le \sum\limits_{j=1}^{N_{vi}} \phi_j(\bx) \le 1,
\]
we have
\[
   M_{ii,lump}
      \ge \int\limits_\Omega \phi^2_i(\bx) \dx
      = \sum_{K \in \omega_i} \int_K \phi_i^2(\bx) \dx
      = \sum_{K \in \omega_i}  \frac{2 \Abs{K} }{(d+1) (d+2)}
      = \frac{2 \Abs{\omega_i}}{(d+1) (d+2)}
\]
and
\[
   M_{ii, lump}
      \le \int\limits_\Omega \phi_i(\bx) \dx
      = \sum_{K \in \omega_i} \int_K \phi_i(\bx) \dx
      = \sum_{K \in \omega_i}  \frac{\Abs{K} }{d+1}
      = \frac{\Abs{\omega_i}}{d+1}
      .
\]
\end{proof}

%--- Lemma: mass matrix, lumped vs. full --------------------------------------
\begin{lemma1}
\label{thm:mass:matrix:lumped:bounds}
The linear finite element mass matrix $M$ and the lumped mass matrix $M_{lump}$ satisfy
\begin{equation*}
   \frac{1}{d+2} M_{lump} \le M \le \frac{d+2}{2} M_{lump}
   .
\end{equation*}
\end{lemma1}
\begin{proof}
Since $M_D \le M_{lump}$ we get the upper bound directly from \cref{eq:M}.
Combining the lower bound in \cref{eq:M} with the upper bound in \cref{eq:ML} gives
\[
   \frac{1}{d+2} M_{lump}
   \le \frac{1}{(d+2)(d+1)}
      \diag\left( \Abs{\omega_1},\dotsc,\Abs{\omega_{N_{vi}}} \right)
   = \frac{1}{2}M_D
   \le M
   .
\]
\end{proof}

%--- Lemma: stiffness matrix ---------------------------------------------------
\begin{lemma1}[{\cite[Sect.~4]{KamHuaXu13}}]
\label{thm:stiffness:matrix}
The linear finite element stiffness matrix $A$ and its diagonal part $A_D$ satisfy
\begin{equation}
   A \le (d+1) A_D .
   \label{eq:A}
\end{equation}
\end{lemma1}

%--- Lemma: A_jj of the stiffness matrix --------------------------------------
\begin{lemma1}
\label{thm:stiffness:matrix:ii}
Let $\D_K$ be the average of the diffusion matrix $\D$ over $K$,
\[
   \D_K = \frac{1}{\Abs{K}} \int_K \D(\bx) \dx .
\]
Then the diagonal entries of the linear finite element stiffness matrix $A$ are bounded by
\begin{equation}
   C_{\hat\nabla} \sum\limits_{K\in \omega_i} \Abs{K}
      \cdot \lambda_{\min}\bigl(\FDF\bigr)
   \le A_{ii}
   \le C_{\hat\nabla}
      \sum\limits_{K\in \omega_i} \Abs{K}
         \cdot \lambda_{\max} \bigl(\FDF\bigr) ,
   \label{eq:A:fdf}
\end{equation}
where
$
   C_{\hat\nabla}
      = \frac{d}{d+1} {\left( \frac{\sqrt{d+1}}{d!} \right)}^{\frac{2}{d}}
      .
      \label{eq:C:nabla}
$
\end{lemma1}
\begin{proof}
From \cref{eq:AM:i} we have
\[
   A_{ii}
   = \int_\Omega \nabla \phi_i^T \D \nabla \phi_i \dx
   = \sum_{K \in \omega_i} \int_K \nabla \phi_i^T \D \nabla \phi_i \dx
   = \sum_{K \in \omega_i} \Abs{K} \; \nabla \phi_i^T \D_K \nabla \phi_i .
\]
Denote the gradient operator in $\hat{K}$ by $\hat{\nabla} = \frac{\partial}{\partial \bxi }$.
By the chain rule $\nabla = {(F_K')}^{-T} \hat{\nabla}$ and
\begin{align}
   A_{ii}
   &= \sum_{K \in \omega_i} \Abs{K} \;
      \hat{\nabla} \hat{\phi}_i^T \FDF \hat{\nabla} \hat{\phi}_i
   \label{eq:aii:duni}
   \\
   &\le \sum_{K \in \omega_i} \Abs{K} \;
      \hat{\nabla} \hat{\phi}_i^T
      \hat{\nabla} \hat{\phi}_i
      \lambda_{\max}\bigl(\FDF\bigr)
   \notag
   .
\end{align}
Recall that $\hat{K}$ is taken to be equilateral.
Thus, $\hat{\nabla} \hat{\phi}_i^T  \hat{\nabla} \hat{\phi}_i = C_{\hat\nabla} $ for all $ i = 1, \dotsc, d+1$.
Consequently,
\[
   A_{ii}
      \le C_{\hat\nabla} \sum_{K \in \omega_i} \Abs{K} \;
         \lambda_{\max}\bigl(\FDF\bigr) .
\]
Similarly, we can obtain the left inequality of \cref{eq:A:fdf}.
\qquad\end{proof}

\vspace{1.5ex}
%--- Remark: A_jj and Q_\D -----------------------------------------------------
\begin{remark}
\normalfont{%
From \cref{eq:QM:def}, with $\M$ being replaced by $\D^{-1}$, the bound \cref{eq:A:fdf} on $A_{ii}$ can be expressed in terms of the element quality measure $\QD(K)$ as
\begin{equation}
   A_{ii}
      \le  C_{\hat\nabla} h_{\D^{-1}}^{-2}
         \sum\limits_{K \in \omega_i} \Abs{K} \QD(K).
   \label{eq:Aii:Q}
\end{equation}
}
\end{remark}

%--- Remark: \D^{-1}-nonobtuse meshes ------------------------------------------
\begin{remark}[$\D^{-1}$-nonobtuse meshes]
\label{rem:Duniform}
\normalfont{%
Note that \cref{thm:stiffness:matrix} is very general and valid for any given mesh.
It implies that
\begin{equation}
   \lambda_{\max}(A) \le (d+1) \max_i A_{ii} .
   \label{eq:lmaxA:Ajj}
\end{equation}
This bound can be sharpened for some special types of meshes.
For example, if a mesh has no obtuse angles with respect to $\D^{-1}$ then $A$ is an M-matrix (its off-diagonal entries are non-positive) and $\sum_j A_{ij} \ge 0$ for all $i$ (e.g.,\ see the proof of Theorem 2.1 of~\cite{LiHua10}).
From the Gershgorin circle theorem we have
\[
   \lambda_{\max}(A)
      \le \max_i \left(A_{ii} + \sum_{j \neq i} \Abs{A_{ij}}\right)
      = \max_i \left(A_{ii} - \sum_{j \neq i} A_{ij}\right)
      = \max_i \left( 2 A_{ii} - \sum_j A_{ij}\right)
\]
and thus
\begin{equation}
   \lambda_{\max}(A)
      \le 2 \max_i A_{ii} .
   \label{eq:A-3}
\end{equation}
If further the mesh is $\D^{-1}$-uniform, then from \cref{eq:eq:ali:i,eq:Aii:Q} we have
\begin{equation}
   \lambda_{\max}(A)
      \le 2 \max_i A_{ii}
      \le 2 C_{\hat\nabla} h_{\D^{-1}}^{-2}
         \max_i\sum\limits_{K \in \omega_i} \Abs{K} \QD(K)
      = 2 C_{\hat\nabla} h_{\D^{-1}}^{-2} \max_i \Abs{\omega_i}.
   \label{eq:A-4}
\end{equation}
}
\end{remark}

%*** Explicit schemes **********************************************************
\section{Explicit time stepping and the stability condition}
\label{sec:explicit}

In this section we study stability conditions for explicit one-step methods applied to the finite element system \cref{eq:fem:system} and obtain estimates for the maximum time step.

Suppose we are given a constant time step $\tau$.
Then an explicit one-step integration scheme with $s$ stages of order $p$ computes approximations $\bU_n\approx \bU(n\tau)$ from
\begin{equation}
   \bU_n = R(-\tau\,M^{-1}A) \bU_{n-1} ,
   \label{erk:scheme}
\end{equation}
where the stability function $R(z)$ is a polynomial in $z$ and satisfies
\begin{equation}
   R(z)
      = 1 + z + \dotso + \frac{z^p}{p!} + \sum_{i=p+1}^s \alpha_i z^i
      = e^z + \cO\left(z^{p+1}\right) .
   \label{erk:stabfunc}
\end{equation}
Classical explicit one-step methods have severe step size restrictions when solving stiff problems as \cref{eq:fem:system} for $N_{vi}\gg 1$.
An interesting alternative are stabilized explicit Runge-Kutta (RK) methods, which have an extended stability domain along the negative real axis and therefore allow for larger time steps than classical explicit one-step methods.
Parameters $\alpha_{p+1}, \dotsc, \alpha_s \in \R$ in \cref{erk:stabfunc} are chosen such that $\Abs{R(z)} \le 1$ for $z\in [-r_s,0]$ and $r_s>0$ is as large as possible.
Explicit methods have low memory demand and can be considered as a good alternative to implicit methods when the solution of algebraic equations arising from the latter is difficult and/or costly.
Impressive examples and comparison results with VODEPK (a stiff ODE solver with Krylov iterations) are documented in~\cite{HunVer03}.
Commonly used explicit methods include DUMKA, Runge-Kutta-Chebychev (RKC) and the orthogonal Runge-Kutta-Chebychev (ROCK) methods.
A common practical choice is $p=2$, but there exist also DUMKA and ROCK methods of higher order~\cite{HaiWan96}.

In the following we first study stability estimates for the approximate solutions $\bU_n$ obtained from \cref{erk:scheme}, assuming that $M$ is a full mass matrix.
However, the decomposition of a consistent mass matrix as a part of an explicit time integration method is in general not affordable, since an implicit scheme with a much larger step may be performed at the same cost.
Hence, we mainly discuss consequences of lumping the mass matrix as a routine procedure for (linear) finite elements.
Although appropriate mass lumping does not affect the overall accuracy, it is well-known that lumping the consistent mass induces dispersion errors that can affect the quality of the numerical solution.
More generally, we consider symmetric positive definite, surrogate matrices $\tM$ that satisfy
\begin{equation}
   c_1 \tM \le M \le c_2 \tM
\end{equation}
and have nearly the same complexity as the diagonal lumped mass matrix $M_{lump}$.
Correction techniques for the dispersive effects of mass lumping and several efficient choices for $\tM$ can be found in~\cite{GuePas13}.
Note that due to \cref{thm:mass:matrix:lumped:bounds} we have $c_1=1/(d+2)$ and $c_2=(d+2)/2$ for the special case $\tM=M_{lump}$.

%%%%%%%%%%%%%%%%%%%%%%%%%%%%%%%%%%%%%%%%%%%%%%%%%%%%%%%%%%%%
\subsection{Stability of explicit one-step integration schemes}
The investigation of the stability is based on the following main observation: if $B$ is a normal matrix and $R$ is a rational function, then
\begin{equation}
   \Norm{R(B)}_2 = \max_i \Abs{R(\lambda_i(B))}.
\label{stab:l2norm}
\end{equation}
This fundamental relation is a direct consequence of the existence of a factorization $B=Q\,\diag \bigl(\lambda_1(B),\dotsc,\lambda_{N_{vi}}(B) \bigr)\,Q^T$ with a unitary matrix $Q$.

Using the fact that $\MAM$ and $\AMA$ are normal matrices, we can prove the stability of the linear finite element approximation computed with an explicit one-step method.
%
%--- General stability theorem -------------------------------------------------
\begin{theorem}
\label{thm:stability}
For a given explicit one-step method with the polynomial stability function $R$, the linear finite element approximation $u^{h}_n$ satisfies
\[
  \Norm{u^{h}_{n}}_{L^2(\Omega)} \le \Norm{u^{h}_{0}}_{L^2(\Omega)}
  \quad \text{and} \quad
  \NormE{u^{h}_{n}}_{H^1(\Omega)} \le \NormE{u^{h}_{0}}_{H^1(\Omega)},
\]
if  the time step $\tau$ is chosen such that
\[
   \max_i \Abs{R\left(-\tau\lambda_i\left(\MA\right)\right)} \le 1.
\]
\end{theorem}
\begin{proof}
Since $R$ is a polynomial function, we have
\[
   R(-\tau\MA)
      = M^{-\frac{1}{2}} R(-\tau\MAM) M^{\frac{1}{2}}
      = A^{-\frac{1}{2}} R(-\tau\AMA) A^{\frac{1}{2}} .
\]
From this, it is easy to see that \cref{erk:scheme} can be written as
\begin{align}
   M^{\frac{1}{2}} \bU_n &= R(-\tau\MAM) M^{\frac{1}{2}} \bU_{n-1} ,
   \label{eq:scheme4v}
   \\
   A^{\frac{1}{2}} \bU_n &= R(-\tau\AMA) A^{\frac{1}{2}} \bU_{n-1} .
   \label{eq:scheme4w}
\end{align}
Since $M$ and $A$ are symmetric and positive definite, $\MAM$ and $\AMA$ are symmetric and therefore normal.
From \cref{stab:l2norm}, our assumption on the time step and the fact that $\MA$, $\MAM$, and $\AMA$ are similar to each other, we get
\[
   \Norm{R(-\tau\MAM)}_2
   = \Norm{R(-\tau\AMA)}_2
   = \max_i \Abs{R(-\tau\lambda_i(\MA))} \le 1 .
\]
Thus, equations \cref{eq:scheme4v,eq:scheme4w} imply
\begin{align*}
   \Norm{u^{h}_{n}}_{L^2(\Omega)}
   &= \Norm{M^{\frac{1}{2}} \bU_n}_2
   \le \Norm{M^{\frac{1}{2}} \bU_{n-1}}_2
   = \Norm{u^{h}_{n-1}}_{L^2(\Omega)}
   ,
   \\
   \NormE{u^{h}_{n}}_{H^1(\Omega)}
   &= \Norm{A^{\frac{1}{2}} \bU_n}_2
   \le \Norm{A^{\frac{1}{2}} \bU_{n-1}}_2
   =\NormE{u^{h}_{n-1}}_{H^1(\Omega)}
   .
\end{align*}
Successive application of these inequalities yields the assertion.
\qquad\end{proof}

We next consider the case where the linear finite element mass matrix $M$ is replaced by a symmetric positive definite, surrogate matrix $\tM$ of lower complexity.
That means, from now on we compute approximations $\bU_n\approx \bU(n\tau)$ from
\begin{equation}
   \bU_n = R(-\tau\,\tM^{-1}A)  \bU_{n-1} .
   \label{erk:approx:scheme}
\end{equation}
%

%--- Stability of the numerical scheme -----------------------------------------
\begin{theorem}
\label{thm:stability:approx:scheme}
For a given explicit one-step method with the polynomial stability function $R$ and a symmetric positive definite, surrogate matrix $\tM$ that satisfies $c_1 \tM \le M \le c_2 \tM$ for some positive constants $c_1$ and $c_2$, the linear finite element approximation $u^{h}_n$ satisfies
\[
   \Norm{u^{h}_{n}}_{L^2(\Omega)}
      \le \sqrt{\frac{c_2}{c_1}} \; \Norm{u^{h}_{0}}_{L^2(\Omega)}
   \quad \text{and} \quad
   \NormE{u^{h}_{n}}_{H^1(\Omega)}
      \le \NormE{u^{h}_{0}}_{H^1(\Omega)},
\]
if the time step $\tau$ is chosen such that
\[
   \max_i \abs{R(-\tau\lambda_i(\tM^{-1}A))} \le 1\,.
\]
\end{theorem}
\begin{proof}
Replacing $M$ by $\tM$ in the proof of \cref{thm:stability} does not change the arguments and gives
\begin{align*}
   \NormE{u^{h}_{n}}_{H^1(\Omega)}
   = \Norm{A^{\frac{1}{2}} \bU_n}_2
   &\le \Norm{A^{\frac{1}{2}} \bU_{n-1}}_2
   = \NormE{u^{h}_{n-1}}_{H^1(\Omega)},
   \\
   \Norm{\tM^{\frac{1}{2}} \bU_n}_2
   &\le \Norm{\tM^{\frac{1}{2}} \bU_{n-1}}_2 .
\end{align*}
From the first inequality, stability in the energy norm follows.
To derive stability in the $L^2$-norm, we make use of the assumption on $\tM$:
\begin{align*}
   \Norm{u^{h}_{n}}^2_{L^2(\Omega)}
      &= {(\bU_n)}^T M \bU_n
      \le c_2 \, {(\bU_n)}^T \tM \bU_n
      = c_2 \,\Norm{\tM^{\frac{1}{2}} \bU_n}_2
      \le c_2 \,\Norm{\tM^{\frac{1}{2}} \bU_{n-1}}_2
      \\
      &\le \dotso
      \le c_2 \,\Norm{\tM^{\frac{1}{2}} \bU_{0}}_2
      =  c_2 \,{(\bU_0)}^T \tM \bU_0
      \le \frac{c_2}{c_1}\, {(\bU_0)}^T M \bU_0
      = \frac{c_2}{c_1}\, \Norm{u^h_0}^2_{L^2(\Omega)},
\end{align*}
which gives the desired result.
\qquad\end{proof}

In the special case $\tM = M_{lump}$ we have the following result.
\begin{corollary1}
Under the assumptions of \cref{thm:stability:approx:scheme} and $\tM=M_{lump}$, we have
\begin{equation*}
   \Norm{u^{h}_{n}}_{L^2(\Omega)}
   \le \frac{d+2}{\sqrt{2}} \; \Norm{u^h_0}_{L^2(\Omega)}
   \quad \text{and} \quad
   \NormE{u^{h}_{n}}_{H^1(\Omega)}
   \le \NormE{u^h_0}_{H^1(\Omega)} .
\end{equation*}
\end{corollary1}

%%%%%%%%%%%%%%%%%%%%%%%%%%%%%%%%%%%%%%%%%%%%%%%%%%%%%%%%%%%%
\subsection{Estimates on the largest eigenvalue%
   \texorpdfstring{ of $\tM^{-1}A$}{}}

The above results show that the contractivity of any given explicit one-step method is guaranteed if all eigenvalues of $- \tau \tM^{-1}A$ are in the corresponding stability domain $\mathcal{S} = \left\{ z \in \mathbb{C} : \Abs{R(z)} \le 1 \right\}$.
As a consequence, the key to the stability analysis of a given scheme is the estimation of the eigenvalues of $\tM^{-1}A$.
The following theorem provides such an estimate for two choices of $\tM$: $\tM=M$ and $\tM=M_{lump}$.
It turns out that in these cases the largest eigenvalue of $\tM^{-1}A$ is equivalent to the largest ratio between the corresponding diagonal entries of $A$ and $\tM$.
\begin{theorem}
\label{thm:lmax:theorem}
The eigenvalues of $\tM^{-1}A$ with $\tM$ being either $M$ or $M_{lump}$ are real and positive.
Moreover, the largest eigenvalue is bounded by
\begin{equation}
   \max_i \frac{A_{ii}}{\tM_{ii}}
   \le \lambda_{\max} \left( \tM^{-1}A \right)
   \le C_* \max_i \frac{A_{ii}}{\tM_{ii}},
   \label{eq:lmax}
\end{equation}
where $C_*$ is given in \cref{tab:lmax}.
\end{theorem}

%--- All results in a table ---------
\begin{table}[t]
   \caption{$C_*$ in \cref{thm:lmax:theorem}\label{tab:lmax}}
   \begin{tabular}{lcc}
      \toprule%
                       & general meshes       & nonobtuse w.r.t.\ $\D^{-1}$\\
      \midrule%
      $\tM = M$        & $2\left(d+1\right)$  & $4$\\
      $\tM = M_{lump}$ & $d+1$                & $2$\\
      \bottomrule%
   \end{tabular}
\end{table}

\begin{proof}
Since $\tM$ and $A$ are symmetric and positive definite and $\tM^{-1}A$ is similar to the symmetric matrix $\tM^{-\frac{1}{2}} A \tM^{-\frac{1}{2}}$, the eigenvalues of $\tM^{-1}A$ are real and positive.

Using the canonical basis vectors $\be_i$ gives
\begin{align*}
   \lambda_{\max} \bigl( \tM^{-1}A \bigr)
   &= \max_{\bv \neq 0}
      \frac {\bv^T A \bv}
            {\bv^T \tM \bv}
   \ge \max_i \frac {\be_i^T A \be_i}
            {\be_i^T \tM \be_i}
   = \max_i \frac{A_{ii}}{\tM_{ii}}.
\end{align*}
Let us first have a look at the case $\tM=M$.
\Cref{thm:mass:matrix,thm:stiffness:matrix} yield
\begin{equation}
   \lambda_{\max} \left( \MA \right)
       = \max_{\bv \neq 0}
         \frac {\bv^T A \bv}
               {\bv^T M \bv}
      \le \max_{\bv \neq 0}
         \frac {(d+1) \bv^T A_D \bv}
               {\frac{1}{2} \bv^T M_D \bv}
      \label{eq:lmax:halve}
      = 2(d+1) \max_i \frac{A_{ii}}{M_{ii}}.
\end{equation}

For the special case of meshes with nonobtuse angles with respect to $\D^{-1}$, the above bound can be sharpened by replacing the factor $d+1$ in \cref{eq:lmax:halve} with $2$ (see \cref{rem:Duniform}).
If $\tM=M_{lump}$ then the factor $1/2$ in the denominator of \cref{eq:lmax:halve} can be replaced by $1$ since $M_{lump}$ is already diagonal.
\qquad\end{proof}

\vspace{1.5ex}
%--- Example: SRK and the diagonal estimate ------------------------------------
\begin{example}[Stabilized Runge-Kutta methods]
\label{ex:explicit:SRK}
\normalfont{%
The stability region of a stabilized RK method of order $p=1$ with $s$ stages extends along the negative real axis of the complex plane, including the interval $\left[-2s^2, 0\right]$~\cite[p.~31f.]{HaiWan96}.
Thus, the method is stable if all eigenvalues of $-\tau \tM^{-1}A$ are between $-2s^2$ and $0$.
This leads to the stability condition
\begin{equation}
   \tau \le \frac{2s^2}{ \lambda_{\max}\left(\tM^{-1}A \right) }.
   \label{eq:tau:srk:exact}
\end{equation}
Using \cref{thm:lmax:theorem} and noticing that ${\left(\max_i \frac{A_{ii}}{\tM_{ii}} \right)}^{-1} = \min_i \frac{\tM_{ii}}{A_{ii}}$, we obtain a bound for the largest permissible time step $\tau_{\max}$ as
\begin{equation}
   \frac{2s^2}{C_*} \min_i \frac{\tM_{ii}}{A_{ii}}
   \le \tau_{\max} \le
   2s^2 \min_i \frac{\tM_{ii}}{A_{ii}} .
   \label{eq:tau:SRK}
\end{equation}
Clearly, if
\[
   \tau > 2s^2 \min_i \frac{\tM_{ii}}{A_{ii}},
\]
we have $\Abs{R\left(-\tau \lambda_{\max}\left(\tM^{-1}A\right)\right)} > 1$ and the scheme becomes unstable.
In order to guarantee stability, the step size has to be chosen such that
\[
   \tau \le \frac{2s^2}{C_*} \min_i \frac{\tM_{ii}}{A_{ii}}.
\]
Note that here $\tM=M$ or $\tM=M_{lump}$.
}
\end{example}

\vspace{1.5ex}
In practice, a few steps of a nonlinear power method are often sufficient to estimate the spectral radius automatically, especially if the eigenvalues are close to the negative real axis.
However, the power method can degenerate in many ways, so precaution has to be taken and theoretical bounds can be helpful.
Such bounds are also necessary for gaining insight on the effects of mesh geometry on the stability of explicit integration schemes and the maximum allowed step size.
The estimate in \cref{thm:lmax:theorem} is easily computable but it does not explain how the mesh geometry affects the time step.
To reveal these effects, we provide several geometric formulations of the estimate in the following.
First, substituting \cref{eq:M,eq:Aii:Q} for $\tM_{ii}$ and $A_{ii}$ in \cref{thm:lmax:theorem} gives the following corollary.

%-- Corollary: geometric estimate in terms of element shape etc. ---------
\begin{corollary1}
\label{thm:lmax:geo}
The largest eigenvalue of $\tM^{-1}A$ is bounded by
\begin{align}
   \lambda_{\max}(\tM^{-1}A)
   & \le 
   %\frac{C_* C_{\hat\nabla} (d+1) (d+2) }{2 }
   C_* C_{\#}
      \max_i \sum\limits_{K \in \omega_i}
         \frac{\Abs{K}}{\Abs{\omega_i}} \Norm{\FDF}_2
   \label{eq:stability:geo}
   \\
   &=
      %\frac{C_* C_{\hat\nabla} (d+1) (d+2) }{2 h_{\D^{-1}}^{2}}
      C_* C_{\#} h_{\D^{-1}}^{-2}
      \max_i \sum\limits_{K \in \omega_i}
         \frac{\Abs{K}}{\Abs{\omega_i}} \QD(K)
   \label{eq:stability:geo-2}
   ,
\end{align}
where $C_{\#} = \frac{1}{2} C_{\hat\nabla} (d+1) (d+2)$, $C_{\hat\nabla}$ and  $C_*$ are given in \cref{eq:C:nabla,tab:lmax} and the element quality $\QD(K)$ is defined in \cref{eq:QM:def} (with $\mathbb{M}$ being replaced by $\mathbb{D}^{-1}$).
\end{corollary1}

The factor $h_{\D^{-1}}^{-2}$ in \cref{eq:stability:geo} corresponds to $h^2$ in the classic stability condition $\tau \sim h^2$ for uniform meshes with the Laplace operator.
Since
\[
   h_{\D^{-1}}
      = {(\Abs{\Omega}_{\D^{-1},h}/N)}^{\frac 1 d }
      \to {(\Abs{\Omega}_{\D^{-1}}/N)}^{\frac 1 d }
\]
as the mesh is being refined, $h_{\D^{-1}}$ can be considered independent of the mesh geometry and therefore it essentially depends only on $N$, $\D$, and $\Omega$.

The mesh geometry effect is reflected mainly through the patch-average of $\Norm{\FDF}_2$ or, alternatively, the element quality measure $\QD(K)$.
Recall from \cref{eq:QM:geometric} that $\QD(K)$ can be seen as a ratio of the average element size to the diameter of the largest sphere inscribed in $K$, both measured in the metric $\D_K^{-1}$.

Hence, we can conclude that the largest possible time step depends on \emph{the number of mesh elements} and \emph{the correspondence of the geometry of the mesh elements to $\D^{-1}$}.
In other words, it is not the mesh geometry itself but \emph{the mesh geometry in relation to the diffusion matrix that matters for the stability of explicit schemes.}

We now study the situation with an $\M$-uniform mesh for a general metric tensor $\M$.
Recall that such a mesh satisfies \cref{eq:muniform:fmf}, which can be rewritten as
\[
   {(F_K')}^{-T} {(F_K')}^{-1} =  h_{\M}^{-2} \M_K
   \qquad \forall K \in \Th.
\]
Then,
\[
   \QD(K)
      = h_{\D^{-1}}^2 \Norm{\FDF}_2
      = {\left( \frac{h_{\D^{-1}}}{h_{\M}} \right)}^2
         \Norm{\M_K \D_K}_2.
\]
Inserting this into \cref{eq:stability:geo} we get
\begin{equation}
   \lambda_{\max}(\tM^{-1}A)
   %\le \frac{C_* C_{\hat\nabla} (d+1) (d+2) }{2 h_{\M}^{2}}
   \le C_* C_{\#} h_{\M}^{-2}
      \max_i \sum\limits_{K \in \omega_i}
      \frac{\Abs{K}}{\Abs{\omega_i}} \cdot \Norm{\M_K \D_K}_2 .
   \label{eq:stability:geo:ii}
\end{equation}
Once again, this shows that the largest eigenvalue of $\tM^{-1}A$ and, consequently, the largest permissible time step depend on the number of elements and the matching between the mesh (essentially determined by $\M$) and the diffusion matrix.
If mesh adaptation and the major diffusion directions match, the largest permissible time step depends mainly on the number of mesh elements.
A mismatch between (anisotropic) mesh adaptation and the diffusion directions can lead to a drastic reduction of the time step (see~\cref{ex:ZhuDu} in \cref{sec:numerical:examples}).
In particular, it implies that one gets both accuracy and stability with the same grid if the solution anisotropy is in correspondence with the diffusion and one would have to trade off accuracy for stability if the demands of accuracy and stability contradict each other (see also remarks by Shewchuk~\cite[Sect.~4.3]{She02a}).
To some extent, the demands of accuracy and stability can be combined using a metric tensor in the form
\[
   \M_K = \theta_K \D_K^{-1} \quad \forall K \in \Th,
\]
where $\theta_K$ is a scalar function based on some (isotropic) error estimate; a similar idea has been used in~\cite{LiHua10} to combine mesh adaptation with satisfaction of the maximum principle.
This will not provide full mesh adaptation but, at least, some degree of it.

%--- D-uniform meshes -----------------------------------------------
\vspace{1.5ex}%
\begin{remark}[Special cases]%
\label{rem:special:cases}%
\normalfont{%
For a uniform mesh ($\M = I$), we have 
\[
   \lambda_{\max}(\tM^{-1}A)
   \le  C_* C_{\#} h^{-2} 
      \max_i \sum\limits_{K \in \omega_i}
      \frac{\Abs{K}}{\Abs{\omega_i}} \cdot \Norm{\D_K}_2
      \approx  C_* C_{\#} h^{-2} 
      \max_i \Norm{\D_{\omega_i}}_2
      ,
\]
where $\D_{\omega_i}$ denotes the average of $\D$ over a patch $\omega_i$.

In case of coefficient-adaptive ($\D^{-1}$-uniform) meshes ($\M = \D^{-1}$), mesh adaptation and diffusion directions match exactly
and \cref{eq:muniform:fmf,eq:aii:duni} yield $A_{ii} = C_{\hat{\nabla}} \Abs{\omega_i} / h_{\D^{-1}}^2$.
Thus, using \cref{eq:M,thm:lmax:theorem} gives
\[
   \lambda_{\max}(\tM^{-1}A) \sim h_{\D^{-1}}^{-2}
   \sim N^{\frac{2}{d}}
   %, \quad \text{or, equivalently,} \quad
   %\lambda_{\max}(\tM^{-1}A) \sim N^{\frac{2}{d}}
   .
\]
}%
\end{remark}

%\vspace{1.5ex}
\begin{remark}[Comparison to results available in the literature]%
\label{rem:zhudu}%
\normalfont{%
For the full mass matrix Zhu and Du~\cite[Theorem~3.1]{ZhuDu14} developed an estimate in terms of the element geometry and the eigenvalues of the diffusion matrix, which is valid for $d \ge 2$ and $P_k$ finite elements.
For the linear finite elements it becomes
\begin{align}
   \frac{\max\limits_K \lambda_{\min}(\D_K) \bZ_K}{d(1 + c_1 p_{\max} (d+2) )}
   &\le \lambda_{\max}\left(\MA\right) \le
   (d+2) \max\limits_K \lambda_{\max}(\D_K) \bZ_K,
   \label{eq:lmax:ZhuDu} 
   \\
   \bZ_K & = \frac{d+1}{d^2} \sum_{i_K} \frac{ \Abs{V_{i_K}}^2}{\Abs{K}^2},
\notag
\end{align}
where $\Abs{V_{i_K}}$ is the volume of a $(d-1)$-dimensional face opposing the $i_K$-th vertex of $K$, $p_{\max}$ is the maximum number of elements in a patch, and $c_1$ is the maximum ratio between the volumes of neighboring elements.
The ratio of the upper bound to the lower one is approximately $d {(d+2)}^2 c_1 p_{\max} \kappa(\D)$, where $\kappa(\D) = \lambda_{\max}(\D_K) / \lambda_{\min}(\D_K)$.

Geometric bound \cref{eq:stability:geo} is similar to \cref{eq:lmax:ZhuDu} but there is a significant difference.
Since $\bZ_K \sim \norm{\FF}_2$, 
%Using our notation,
the interplay between the mesh geometry and the diffusion matrix in \cref{eq:stability:geo,eq:lmax:ZhuDu} is mainly reflected by
\[
   \Norm{\FDF}_2
   \quad \text{and} \quad
   \lambda_{\max}(\D_K) \Norm{\FF}_2 ,
\]
respectively.
If either the mesh or $\D$ are isotropic then the factors are comparable.
However, if both the mesh and $\D$ are anisotropic then the former factor can be much smaller than the latter.
In the worst situation, the accuracy of \cref{eq:lmax:ZhuDu} can deteriorate proportionally to $\kappa(\D)$ (see \cref{ex:aniso} in \cref{sec:numerical:examples}), whereas the bound \cref{eq:lmax} in \cref{thm:lmax:theorem} in terms of matrix entries is sharp within a factor of at most $2 (d+1)$, independently of the mesh and $\D$.

In the case of mass lumping, Shewchuk~\cite[Sect.~3]{She02a} obtained geometric bounds in two and three dimensions. The bounds can be generalized to any dimension as
\begin{equation}
   \frac{1}{d} \max_K \bS_K \leq \lambda_{\max}\left(\tM^{-1}A\right) \leq p_{\max} \max_K \bS_K,
   \quad 
   \bS_K =  \frac{1}{d^2} \sum_{i_K} 
      \frac{\Abs{K}}{\tilde{M}_{i_K i_K}} 
      \frac{\abs{V_{i_K}}_{\D^{-1}}^2}{\Abs{K}_{\D^{-1}}^2 } ,
   \label{eq:lmax:Shewchuk}
\end{equation}
where $\Abs{V_{i_K}}_{\D^{-1}}$ is the volume of a $(d-1)$-dimensional face opposing the $i_K$-th vertex of $K$ with respect to $\D^{-1}$ and $\tilde{M}_{i_K i_K}$ is the entry of the (global) lumped mass matrix corresponding to the node $i_K$.
The bound takes the interplay between the mesh shape and $\D$ fully into account and is tight within a factor of $d p_{\max}$, independently of $\D$, but it still has a weak mesh dependence through $p_{\max}$ (typically, $p_{\max} \ge 6$ in 2D and can be much larger in higher dimensions).
Numerical examples in \cref{sec:numerical:examples} show that it is comparable but less accurate than bound \cref{eq:lmax} obtained in this paper.%

For the lumped case there is also an earlier result by Zhu and Du~\cite[Theorem~3.1]{ZhuDu11} but we omit it in this study since it is less accurate than Shewchuk's bound \cref{eq:lmax:Shewchuk}.% and Zhu and Du's more recent bound \cref{eq:lmax:ZhuDu}.%
}%
\end{remark}

%*** Numerics **************************************************************
\section{Numerical examples}
\label{sec:numerical:examples}

To test the developed estimates we continue \cref{ex:explicit:SRK} (stabilized Runge-Kutta methods) and compare the exact value of the largest permissible time step \cref{eq:tau:srk:exact} with the lower bound \cref{eq:tau:SRK},
\[
   \tau_{\max} = \frac{2s^2}{\lambda_{\max}(\MA)}
   \quad \text{and} \quad
   \tau_h = \frac{2s^2}{C_*} \min_i \frac{M_{ii}}{A_{ii}},
\]
and compute the ratio $\tau_{\max} / \tau_h$ to evaluate the accuracy of the estimate.
Since $\tau_{\max} / \tau_h$ is independent of the number of stages $s$, we rescale the values of $\tau_{\max}$ and $\tau_h$ by $s^{-2}$ to stay general, i.e.,\ in the following we compare
\begin{equation}
   \frac{\tau_{\max}}{s^2} = \frac{2}{\lambda_{\max}(\MA)}
   \quad \text{with} \quad
   \frac{\tau_h}{s^2} = \frac{2}{C_*} \min_i \frac{M_{ii}}{A_{ii}}.
   \label{eq:tauh:mat}
\end{equation}
Note that \cref{eq:tau:SRK} implies that $1 \le \tau_{\max} / \tau_h \le C_*$ for any mesh and any diffusion matrix $\D$.
Moreover, from \cref{eq:stability:geo:ii},
\begin{equation}
   \frac{\tau_{\max}}{s^2}
      \ge \frac{2 h_{\M}^{2}}
      {C_* C_{\#} \max_i \frac{1}{\Abs{\omega_i}}
      \sum\limits_{K \in \omega_i} \Abs{K}\cdot \Norm{\M_K \D_K}_2} .
\label{eq:tauh:geo}
\end{equation}

%-------------------------------------------------------------------------------
%--- 1D example with \D from PetSau12 ------------------------------------------
%-------------------------------------------------------------------------------
\begin{example}[1D example~{\cite[Sects.~6.1 and~6.2]{PetSau12}}]
\label{sec:PetSau}
\normalfont{%
As a first example we consider the heat diffusion $u_t = {(\D u_x)}_x$ in $\Omega= (0,1)$ with the diffusion coefficients
\[
   \D(x)
      = {\left(
         2 - \sin\left( 2 \pi \frac{x}{\varepsilon} \right)
      \right)}^{-1}
   \quad \text{and} \quad
   \D(x)
      = {\left(
         2 - \sin\left( 2 \pi \tan \frac{(1-\varepsilon)\pi x}{2} \right)
      \right)}^{-1} ,
\]
where $\varepsilon$ is a positive parameter (\cref{fig:PetSau}).
We choose $\varepsilon = 2^{-4}$ for our tests.

%--- D for PetSau -------------------------------------------------
\begin{figure}[t]
   \subcaptionbox{periodic\label{fig:PetSau:D:per}}[0.5\textwidth]
   {%
      \tikzsetnextfilename{hkl2013-ps-per-D}%
      \centering%
      \begin{tikzpicture}%
         \begin{axis}[%
               %axis equal,
               ymin=0.25, ymax=1.075,
               xmin=-0.05, xmax=1.05,
               width=0.45\linewidth,
               height=0.32\linewidth,
               %ytick pos=right,
            ]
            % for best plot results (to match the period) choose
            % samples = 1 + 64 * n, n >= 1
            \addplot[smooth, domain=0:1, samples=65]
               {1 / ( 2 - sin(360*x*16) )};
       \end{axis}%
   \end{tikzpicture}%
   }%
   \subcaptionbox{nonperiodic\label{fig:PetSau:D:nonp}}[0.5\textwidth]
   {%
      \tikzsetnextfilename{hkl2013-ps-nonp-D}%
      \centering%
      \begin{tikzpicture}%
         \begin{axis}[%
               %axis equal,
               ymin=0.25, ymax=1.075,
               xmin=-0.05, xmax=1.05,
               width=0.45\linewidth,
               height=0.32\linewidth,
            ]
            %% don't look good: equidistant nodes
            %\addplot[smooth, domain=0:1, samples1000]
            %   {(1 / ( 2 - sin( 360 * tan( 15*180*x/32 ) ) ))};
            % much better: a parametrized curve,
            % for best results (to match the period) choose
            % samples = 1 + 256 * n, n >= 1
            \addplot[smooth, domain=0:tan(15*180/32), samples=257]
               ({atan(x)*32/15/180}, {(1 / ( 2 - sin( 360 * x) )});
         \end{axis}%
      \end{tikzpicture}%
   }
   \caption{Diffusion coefficients $\D$ in 1D (\cref{sec:PetSau})\label{fig:PetSau}}
\end{figure}
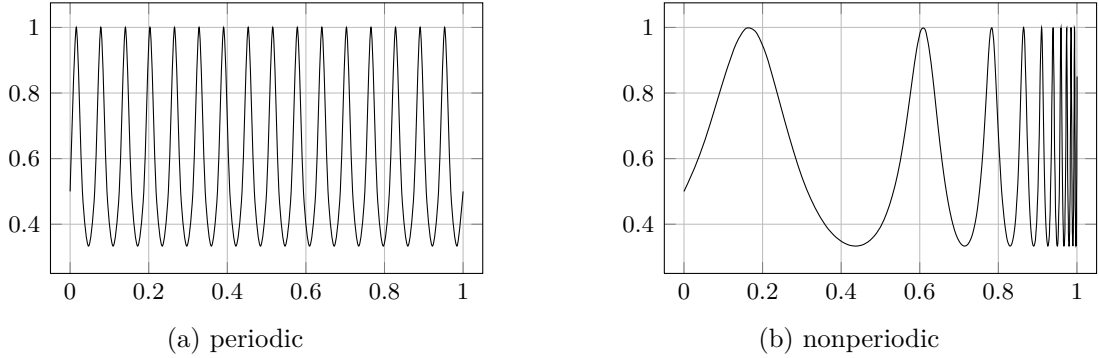

Numerical results in \cref{tab:PetSau} show that $1.00 \le \tau_{\max} / \tau_h \le 1.45$ for all considered meshes and cases, which is consistent with the theoretical prediction $1 \le \tau_{\max} / \tau_h \le 2$ (with mass lumping) and $1 \le \tau_{\max} / \tau_h \le 4$ (without mass lumping).
Interestingly, for this example, the estimate appears to be even asymptotically exact ($\tau_{\max} / \tau_h \to 1$ as $N \to \infty$) except for the case of $\D^{-1}$-uniform meshes with mass lumping.

\Cref{tab:PetSau} further shows that in case of mass lumping  $\tau_{\max}$ is roughly three times as large as $\tau_{\max}$ without mass lumping.
The largest permissible time step $\tau_{\max}$ for $\D^{-1}$-uniform meshes is approximately \numrange{1.4}{1.8} times as large as for uniform meshes.
}\vspace{1.5ex}
\end{example}

\afterpage{%
%--- Table for PetSau -------------------------------------------------
\begin{table}[t]
   \caption{Numerical results in 1D (\cref{sec:PetSau}\label{tab:PetSau})}
   %--- periodic D (PetSau12) --------------------------
   \begin{subtable}[c]{1.0\linewidth}
      \centering
      \subcaption{periodic $\D$ (\cref{fig:PetSau:D:per})\label{tab:PetSau:periodic}}
      \begin{tabular}[b]{@{} r llr llr @{}}
         \toprule%
            & \multicolumn{3}{c}{with mass lumping}
            & \multicolumn{3}{c}{without mass lumping}\\
         \cmidrule(r){2-4}
         \cmidrule(r){5-7}
            $N\phantom{12}$
            & $\frac{\tau_{\max}}{s^2}$
            & $\frac{\tau_h}{s^2}$
            & $\frac{\tau_{\max}}{\tau_h}$
            & $\frac{\tau_{\max}}{s^2}$
            & $\frac{\tau_h}{s^2}$
            & $\frac{\tau_{\max}}{\tau_h}$ \\
         \midrule%
         \multicolumn{7}{c}{uniform meshes} \\
         \midrule%
            %-- N + lumped(tau_h/s^2 - tau/s^2 - tau/tau_h) + non-lumped(...)
\num{   64} & \num{1.84e-04} & \num{1.27e-04} & 1.45 & \num{6.57e-05} & \num{5.10e-05} & 1.29 \\
\num{  128} & \num{3.79e-05} & \num{3.26e-05} & 1.16 & \num{1.45e-05} & \num{1.09e-05} & 1.32 \\
\num{  256} & \num{8.66e-06} & \num{7.76e-06} & 1.12 & \num{3.11e-06} & \num{2.59e-06} & 1.20 \\
\num{  512} & \num{2.04e-06} & \num{1.91e-06} & 1.06 & \num{7.08e-07} & \num{6.37e-07} & 1.11 \\
\num{ 1024} & \num{4.93e-07} & \num{4.77e-07} & 1.03 & \num{1.68e-07} & \num{1.59e-07} & 1.06 \\
\num{ 2048} & \num{1.21e-07} & \num{1.19e-07} & 1.02 & \num{4.09e-08} & \num{3.97e-08} & 1.03 \\

         \midrule%
         \multicolumn{7}{c}{$\D^{-1}$-uniform meshes} \\
         \midrule%
            %-- N + lumped(tau_h/s^2 - tau/s^2 - tau/tau_h) + non-lumped(...)
\num{   64} & \num{2.30e-04} & \num{1.86e-04} & 1.23 & \num{7.67e-05} & \num{7.54e-05} & 1.02 \\
\num{  128} & \num{5.86e-05} & \num{4.86e-05} & 1.21 & \num{1.96e-05} & \num{1.94e-05} & 1.01 \\
\num{  256} & \num{1.47e-05} & \num{1.22e-05} & 1.21 & \num{4.91e-06} & \num{4.90e-06} & 1.00 \\
\num{  512} & \num{3.69e-06} & \num{3.06e-06} & 1.21 & \num{1.23e-06} & \num{1.23e-06} & 1.00 \\
\num{ 1024} & \num{9.22e-07} & \num{7.67e-07} & 1.20 & \num{3.07e-07} & \num{3.07e-07} & 1.00 \\
\num{ 2048} & \num{2.31e-07} & \num{1.92e-07} & 1.20 & \num{7.68e-08} & \num{7.68e-08} & 1.00 \\

         \bottomrule%
      \end{tabular}
      \vspace{3ex}
   \end{subtable}
   %--- non-peiodic D (PetSau12) --------------------------
   \begin{subtable}[c]{1.0\linewidth}
      \centering
      \subcaption{nonperiodic $\D$
         (\cref{fig:PetSau:D:nonp})\label{tab:PetSau:nonperiodic}}
      \begin{tabular}[b]{@{} r llr llr @{}}
         \toprule%
            & \multicolumn{3}{c}{with mass lumping}
            & \multicolumn{3}{c}{without mass lumping}\\
         \cmidrule(r){2-4}
         \cmidrule(r){5-7}
            $N\phantom{12}$
            & $\frac{\tau_{\max}}{s^2}$
            & $\frac{\tau_h}{s^2}$
            & $\frac{\tau_{\max}}{\tau_h}$
            & $\frac{\tau_{\max}}{s^2}$
            & $\frac{\tau_h}{s^2}$
            & $\frac{\tau_{\max}}{\tau_h}$ \\
         \midrule%
         \multicolumn{7}{c}{uniform meshes} \\
         \midrule%
            %-- N + lumped(tau_h/s^2 - tau/s^2 - tau/tau_h) + non-lumped(...)
\num{   64} & \num{1.25e-04} & \num{1.19e-04} & 1.05 & \num{4.31e-05} & \num{3.96e-05} & 1.09 \\
\num{  128} & \num{3.09e-05} & \num{3.01e-05} & 1.03 & \num{1.05e-05} & \num{1.00e-05} & 1.04 \\
\num{  256} & \num{7.67e-06} & \num{7.57e-06} & 1.01 & \num{2.58e-06} & \num{2.52e-06} & 1.02 \\
\num{  512} & \num{1.91e-06} & \num{1.90e-06} & 1.01 & \num{6.41e-07} & \num{6.33e-07} & 1.01 \\
\num{ 1024} & \num{4.78e-07} & \num{4.76e-07} & 1.00 & \num{1.60e-07} & \num{1.59e-07} & 1.01 \\
\num{ 2048} & \num{1.19e-07} & \num{1.19e-07} & 1.00 & \num{3.98e-08} & \num{3.97e-08} & 1.00 \\

         \midrule%
         \multicolumn{7}{c}{$\D^{-1}$-uniform meshes} \\
         \midrule%
            %-- N + lumped(tau_h/s^2 - tau/s^2 - tau/tau_h) + non-lumped(...)
\num{   64} & \num{2.04e-04} & \num{1.68e-04} & 1.22 & \num{7.09e-05} & \num{6.59e-05} & 1.08 \\
\num{  128} & \num{5.28e-05} & \num{4.18e-05} & 1.26 & \num{1.82e-05} & \num{1.67e-05} & 1.09 \\
\num{  256} & \num{1.32e-05} & \num{1.10e-05} & 1.21 & \num{4.53e-06} & \num{4.24e-06} & 1.07 \\
\num{  512} & \num{3.43e-06} & \num{2.76e-06} & 1.25 & \num{1.15e-06} & \num{1.12e-06} & 1.02 \\
\num{ 1024} & \num{8.65e-07} & \num{6.98e-07} & 1.24 & \num{2.89e-07} & \num{2.86e-07} & 1.01 \\
\num{ 2048} & \num{2.17e-07} & \num{1.77e-07} & 1.22 & \num{7.23e-08} & \num{7.20e-08} & 1.00 \\

         \bottomrule%
      \end{tabular}
   \end{subtable}
\end{table}
\clearpage{}
}

%-------------------------------------------------------------------------------
%--- 2D, \D = I ---
%-------------------------------------------------------------------------------
\begin{example}[2D example, $\D = I$]
\label{ex:ZhuDu}
\normalfont{%
In this example we consider the most simple case of $\D=I$.
Mesh examples are taken from~\cite{ZhuDu11,ZhuDu14}; they are: uniform isotropic, uniform anisotropic and strongly refined towards the boundary (\cref{fig:ZhuDu:mesh}).
Since these meshes have no obtuse angles, we can use sharper bounds with $C_* = 2$ (mass lumping) or $C_* = 4$ (no mass lumping) and therefore we expect that $1 \le \tau_{\max} / \tau_h \le 2$ or $1 \le \tau_{\max} / \tau_h \le 4$, respectively.

\Cref{tab:square:meshes} shows that $1.14 \le \tau_{\max} / \tau_h \le 1.69$ (mass lumping) and $1.18 \le \tau_{\max} / \tau_h \le 2.33$ (no mass lumping).
In comparison, the same ratio if using \cref{eq:lmax:ZhuDu} and \cref{eq:lmax:Shewchuk} ranges from %\numrange{12.00}{19.48}.
\numrange{1.78}{3.50}%
\footnote{In our tests, the estimate by Zhu and Du~\cite{ZhuDu14} seems to provide better results than in the numerical examples of the original paper.}
and \numrange{4.00}{6.77}, respectively.
In this example $\D = I$, so that the difference is mainly due to the fact that estimates in terms of mesh geometry are generally less tight than those in terms of matrix entries since additional estimation steps decrease the accuracy.
%\changed{\newline> > > OBSOLETE now both give the same $1/12N$ > > > Nevertheless, the new estimate \cref{eq:stability:geo} (or \cref{eq:tauh:geo}) in terms of mesh geometry is more accurate than those in~\cite{ZhuDu11,ZhuDu14}:
%for the explicit Euler scheme with uniform meshes (\cref{fig:ZhuDu:mesh:iso}) and $\D = I$ the estimate~\cite[Eq.~6.1]{ZhuDu14} yields
%%\[
%$
%   \tau_{\max} \ge \frac{1}{24 N} % \le \tau_{\max} \le 8\frac{1}{3} N^{-1},
%$
%%\]
%whereas \cref{eq:tauh:geo} gives
%%\[
%$
%   \tau_{\max} \ge \frac{1}{12 N}. %\le \tau_{\max} \le 2 N^{-1}.
%%\]
%$
%< < <}

Notice the significant reduction of $\tau_{\max}$ when the mesh gets adapted in the ``wrong'' way, i.e.,\ away from $\D^{-1}$.
For example%(\cref{tab:square:meshes})
, a $32 \times 32$ uniform mesh requires $\tau_{\max} = \num{2.38e-4}$, whereas the $4\times 256$ mesh with the same number of elements requires $\tau_{\max} = \num{6.36e-6}$, a reduction by a factor of $37$.
A strongly anisotropic $4\times16$ mesh adapted towards the boundary with a much smaller number of elements leads to the further reduction of the step size by a factor of \num{3000}.
Thus, the matching between the element geometry and the diffusion matrix has significant effects on the time step size and, depending on the anisotropy of the mesh and diffusion matrix, changes in the mesh alignment can result in changes in the time step size by orders of magnitude.

Again, mass lumping allows approximately \numrange{1.9}{3.2} times larger time steps.

%--- ZhuDu example meshes -------------------------------------------------
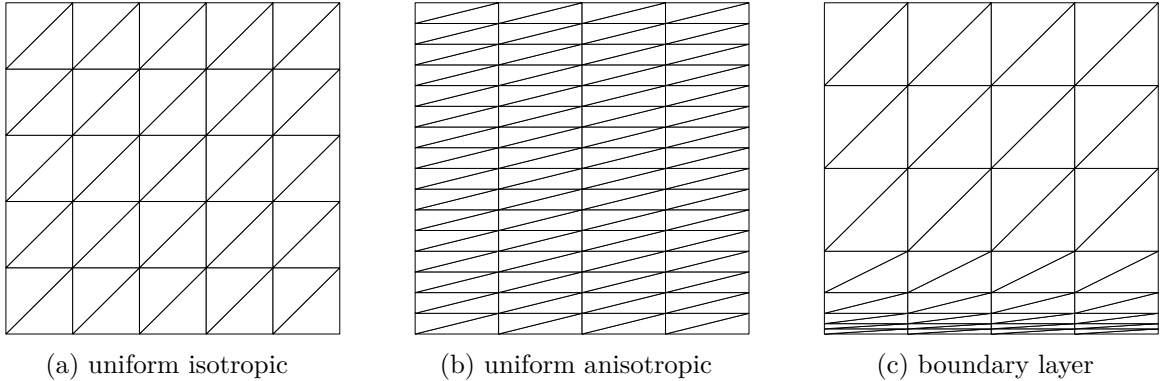
\begin{figure}[t]
   \hfill{}%
   \subcaptionbox{uniform isotropic\label{fig:ZhuDu:mesh:iso}}[0.33\textwidth]
   {%
      \tikzsetnextfilename{hkl2013-zd-uni-iso-mesh}%
      \centering%
      \begin{tikzpicture}[scale=0.9]
         \begin{axis}[axis equal, hide axis,
               xmin = 0.0, xmax = 1.0, ymin = 0.0, ymax = 1.0,
               colormap={bw}{gray(0cm)=(0); gray(1cm)=(0)},
               height=0.40\textwidth,
            ]
            \addplot[patch, fill=white, patch table = {hkl2013-zd-uni-iso-elements.dat}]
               table [x index=0, y index=1] {hkl2013-zd-uni-iso-nodes.dat};
         \end{axis}%
      \end{tikzpicture}%
   }%
   \hfill{}%
   \subcaptionbox{uniform anisotropic\label{fig:ZhuDu:mesh:ani}}[0.33\textwidth]
   {%
      \tikzsetnextfilename{hkl2013-zd-uni-ani-mesh}%
      \centering%
      \begin{tikzpicture}[scale=0.9]
         \begin{axis}[axis equal, hide axis,
               xmin = 0.0, xmax = 1.0, ymin = 0.0, ymax = 1.0,
               colormap={bw}{gray(0cm)=(0); gray(1cm)=(0)},
               height=0.40\textwidth,
            ]
            \addplot[patch, fill=white, patch table = {hkl2013-zd-uni-ani-elements.dat}]
               table [x index=0, y index=1] {hkl2013-zd-uni-ani-nodes.dat};
         \end{axis}%
      \end{tikzpicture}%
   }%
   \hfill{}%
   \subcaptionbox{boundary layer\label{fig:ZhuDu:mesh:p2}}[0.33\textwidth]
   {%
      \tikzsetnextfilename{hkl2013-zd-p2-mesh}%
      \centering%
      \begin{tikzpicture}[scale=0.9]
         \begin{axis}[axis equal, hide axis,
               xmin = 0.0, xmax = 1.0, ymin = 0.0, ymax = 1.0,
               colormap={bw}{gray(0cm)=(0); gray(1cm)=(0)},
               height=0.40\textwidth,
            ]
            \addplot[patch, fill=white, patch table = {hkl2013-zd-p2-elements.dat}]
               table [x index=0, y index=1] {hkl2013-zd-p2-nodes.dat};
         \end{axis}%
      \end{tikzpicture}%
    }%
   \hfill{}%
   \caption{Mesh examples in 2D (\cref{ex:ZhuDu})\label{fig:ZhuDu:mesh}}
\end{figure}
}\vspace{1.5ex}
\end{example}

\begin{table}[p]
   \caption{Numerical results in 2D (\cref{ex:ZhuDu})\label{tab:square:meshes}}
   \begin{tabular}[b]{@{} lr l lr lr @{}}
      \toprule%
       \multicolumn{3}{c}{without mass lumping}
       & \multicolumn{2}{c}{new estimate \cref{eq:tauh:mat}}
       & \multicolumn{2}{c}{Zhu \& Du~\cite{ZhuDu14}}\\
      %\cmidrule(r){4-5}
      \cmidrule(r){4-5}
      \cmidrule(r){6-7}
         \multicolumn{1}{@{}l}{\phantom{1}mesh}
         & $N\phantom{1}$
         & $\frac{\tau_{\max}}{s^2}$
         & $\frac{\tau_h}{s^2}$
         & $\frac{\tau_{\max}}{\tau_h}$
         & $\frac{\tau_h}{s^2}$
         & $\frac{\tau_{\max}}{\tau_h}$ \\
      \midrule%
      \multicolumn{7}{c}{uniform isotropic (\cref{fig:ZhuDu:mesh:iso})} \\
      \midrule%
         $8 \times 8$   &\num{   128}   &\num{1.31e-03} &\num{9.77e-04} &\num{1.34}   &\num{6.51e-04} &\num{2.01} \\
$16 \times 16$   &\num{   512}   &\num{3.09e-04} &\num{2.44e-04} &\num{1.27}   &\num{1.63e-04} &\num{1.90} \\
$32 \times 32$   &\num{  2048}   &\num{7.60e-05} &\num{6.10e-05} &\num{1.24}   &\num{4.07e-05} &\num{1.87} \\
$64 \times 64$   &\num{  8192}   &\num{1.89e-05} &\num{1.53e-05} &\num{1.24}   &\num{1.02e-05} &\num{1.86} \\
$128 \times 128$   &\num{ 32768}   &\num{4.72e-06} &\num{3.81e-06} &\num{1.24}   &\num{2.54e-06} &\num{1.86} \\

      \midrule%
      \multicolumn{7}{c}{uniform anisotropic (\cref{fig:ZhuDu:mesh:ani})} \\
      \midrule%
         $32 \times 32$   &\num{  2048}   &\num{7.60e-05} &\num{6.10e-05} &\num{1.24}   &\num{4.07e-05} &\num{1.87} \\
$16 \times 64$   &\num{  2048}   &\num{3.40e-05} &\num{2.87e-05} &\num{1.18}   &\num{1.91e-05} &\num{1.78} \\
$8 \times 128$   &\num{  2048}   &\num{9.00e-06} &\num{7.60e-06} &\num{1.18}   &\num{5.07e-06} &\num{1.78} \\
$4 \times 256$   &\num{  2048}   &\num{2.38e-06} &\num{1.91e-06} &\num{1.25}   &\num{1.27e-06} &\num{1.87} \\
$2 \times 512$   &\num{  2048}   &\num{6.36e-07} &\num{4.77e-07} &\num{1.33}   &\num{3.18e-07} &\num{2.00} \\

      \midrule%
      \multicolumn{7}{c}{boundary layer (\cref{fig:ZhuDu:mesh:p2})}\\
      \midrule%
         $4 \times 8$   &\num{    64}   &\num{7.08e-05} &\num{3.04e-05} &\num{2.33}   &\num{2.03e-05} &\num{3.49} \\
$4 \times 10$   &\num{    80}   &\num{4.45e-06} &\num{1.91e-06} &\num{2.33}   &\num{1.27e-06} &\num{3.50} \\
$4 \times 12$   &\num{    96}   &\num{2.78e-07} &\num{1.19e-07} &\num{2.33}   &\num{7.95e-08} &\num{3.50} \\
$4 \times 14$   &\num{   112}   &\num{1.74e-08} &\num{7.45e-09} &\num{2.33}   &\num{4.97e-09} &\num{3.50} \\
$4 \times 16$   &\num{   128}   &\num{1.09e-09} &\num{4.66e-10} &\num{2.33}   &\num{3.10e-10} &\num{3.50} \\

      \bottomrule%
      \\[-0.313pt]
      \toprule%
       \multicolumn{3}{c}{with mass lumping}
       & \multicolumn{2}{c}{new estimate \cref{eq:tauh:mat}}
       & \multicolumn{2}{c}{Shewchuk~\cite{She02a}}\\
      %\cmidrule(r){1-3}
      \cmidrule(r){4-5}
      \cmidrule(r){6-7}
         \multicolumn{1}{@{}l}{\phantom{1}mesh}
         & $N\phantom{1}$
         & $\frac{\tau_{\max}}{s^2}$
         & $\frac{\tau_h}{s^2}$
         & $\frac{\tau_{\max}}{\tau_h}$
         & $\frac{\tau_h}{s^2}$
         & $\frac{\tau_{\max}}{\tau_h}$ \\
      \midrule%
      \multicolumn{7}{c}{uniform isotropic (\cref{fig:ZhuDu:mesh:iso})} \\
      \midrule%
         $8 \times 8$   &\num{   128}   &\num{3.79e-03} &\num{2.60e-03} &\num{1.46}   &\num{6.51e-04} &\num{5.82} \\
$16 \times 16$   &\num{   512}   &\num{9.53e-04} &\num{6.51e-04} &\num{1.46}   &\num{1.63e-04} &\num{5.86} \\
$32 \times 32$   &\num{  2048}   &\num{2.38e-04} &\num{1.63e-04} &\num{1.46}   &\num{4.07e-05} &\num{5.86} \\
$64 \times 64$   &\num{  8192}   &\num{5.96e-05} &\num{4.07e-05} &\num{1.46}   &\num{1.02e-05} &\num{5.86} \\
$128 \times 128$   &\num{ 32768}   &\num{1.49e-05} &\num{1.02e-05} &\num{1.46}   &\num{2.54e-06} &\num{5.86} \\

      \midrule%
      \multicolumn{7}{c}{uniform anisotropic (\cref{fig:ZhuDu:mesh:ani})} \\
      \midrule%
         $32 \times 32$   &\num{  2048}   &\num{2.38e-04} &\num{1.63e-04} &\num{1.46}   &\num{4.07e-05} &\num{5.86} \\
$16 \times 64$   &\num{  2048}   &\num{9.86e-05} &\num{7.66e-05} &\num{1.29}   &\num{1.91e-05} &\num{5.15} \\
$8 \times 128$   &\num{  2048}   &\num{2.54e-05} &\num{2.03e-05} &\num{1.25}   &\num{5.07e-06} &\num{5.01} \\
$4 \times 256$   &\num{  2048}   &\num{6.36e-06} &\num{5.09e-06} &\num{1.25}   &\num{1.27e-06} &\num{5.00} \\
$2 \times 512$   &\num{  2048}   &\num{1.27e-06} &\num{1.11e-06} &\num{1.14}   &\num{3.18e-07} &\num{4.00} \\

      \midrule%
      \multicolumn{7}{c}{boundary layer (\cref{fig:ZhuDu:mesh:p2})}\\
      \midrule%
         $4 \times 8$   &\num{    64}   &\num{1.37e-04} &\num{8.11e-05} &\num{1.69}   &\num{2.03e-05} &\num{6.77} \\
$4 \times 10$   &\num{    80}   &\num{8.61e-06} &\num{5.09e-06} &\num{1.69}   &\num{1.27e-06} &\num{6.77} \\
$4 \times 12$   &\num{    96}   &\num{5.38e-07} &\num{3.18e-07} &\num{1.69}   &\num{7.95e-08} &\num{6.77} \\
$4 \times 14$   &\num{   112}   &\num{3.36e-08} &\num{1.99e-08} &\num{1.69}   &\num{4.97e-09} &\num{6.77} \\
$4 \times 16$   &\num{   128}   &\num{2.10e-09} &\num{1.24e-09} &\num{1.69}   &\num{3.10e-10} &\num{6.77} \\

      \bottomrule%
   \end{tabular}%
\end{table}

%-------------------------------------------------------------------------------
%--- 2D, KARDOS, all bells & wistles
%-------------------------------------------------------------------------------
\begin{example}[2D groundwater flow with jumping coefficients~\cite{MicPer08}]
\label{ex:waterflow}
\normalfont{%
As the next example we consider groundwater flow through an aquifer.
The problem is given by the IBVP~\cref{eq:IBVP} with $\Omega = (0,100) \times (0,100)$ and two impermeable subdomains $\Omega_1 = (0,80) \times (64,68)$ and $\Omega_2 = (20,100) \times (40,44)$.
\Cref{fig:waterflow:domain} shows the diffusion matrix $\D$ and the boundary conditions.
Although $\D$ is isotropic, it has a jump between the subdomains, leading to the anisotropic behavior of the solution.

We compute the solution by $h$-refinement in the standard way and use Hessian recovery based mesh adaptation to obtain adaptive meshes at particular time points and compare the exact $\tau_{\max}$ with the lower bound $\tau_h$.
For our computation we used \texttt{KARDOS}~\cite{ErdLanRoi02} to solve the PDE and \texttt{BAMG}~\cite{bamg} for mesh generation.
Examples of adaptive meshes are shown in \cref{fig:waterflow} for the time points $t=\num{1.0e4}$ and $t=\num{1.0e5}$.
Note that these meshes have oblique elements and angles close to \ang{180}: the maximum angles in \cref{fig:waterflow:iv,fig:waterflow:v} are \ang{175} and \ang{177}, respectively.

\Cref{tab:waterflow} shows that the ratio $\tau_{\max} / \tau_h$ is about \numrange{2.13}{2.48} with mass lumping and \numrange{3.25}{3.87} without mass lumping, which is consistent with the theoretical upper bounds $d + 1 = 3$ and $2 (d + 1) = 6$.
In this example, mass lumping would allow \numrange{2.6}{2.8} times larger time steps, which is similar to \cref{ex:ZhuDu} (a factor of \numrange{1.9}{3.2} there).

In a practical computation, however, one would rather use a numerical approximation for $\lambda_{\max}(\MA)$. 
Typically, five steps of the Lánczos method with a random starting vector approximate the largest eigenvalue within 10\%.
Another practical alternative is the power method, for which it is reported~\cite[Sect.~3.2]{SomShaVer98a} that, for the case of eigenvalues being close to the negative real axis, usually only a few iterations are required if the computed eigenvector from the previous step is used as a new starting vector.
To compare it with our theoretical estimate, we additionally computed $\tau_h$ using five steps of the Lánczos method with a random starting vector (divided by 1.1 as a security factor since Lánczos approximation is an approximation from below).
\Cref{tab:waterflow:lanczos} shows that the corresponding ratio $\tau_{\max} / \tau_h$ is about \numrange{1.00}{1.07}, i.e., the computed time step approximation is within 7\% from the optimal value.
In our computations, the accuracy of our theoretical estimate \cref{eq:tauh:mat} corresponds to about two to three steps of the Lánczos method.

We would like to also point out that the lower bound in~\cref{eq:lmax} can be used as a practical security check for a numerical approximation: if the computed numerical approximation of $\lambda_{\max}(\MA)$ is smaller than this bound, the time step is guaranteed to be out of the stability region of the time integration method.
}\vspace{1.5ex}
\end{example}

% waterflow: mesh examples
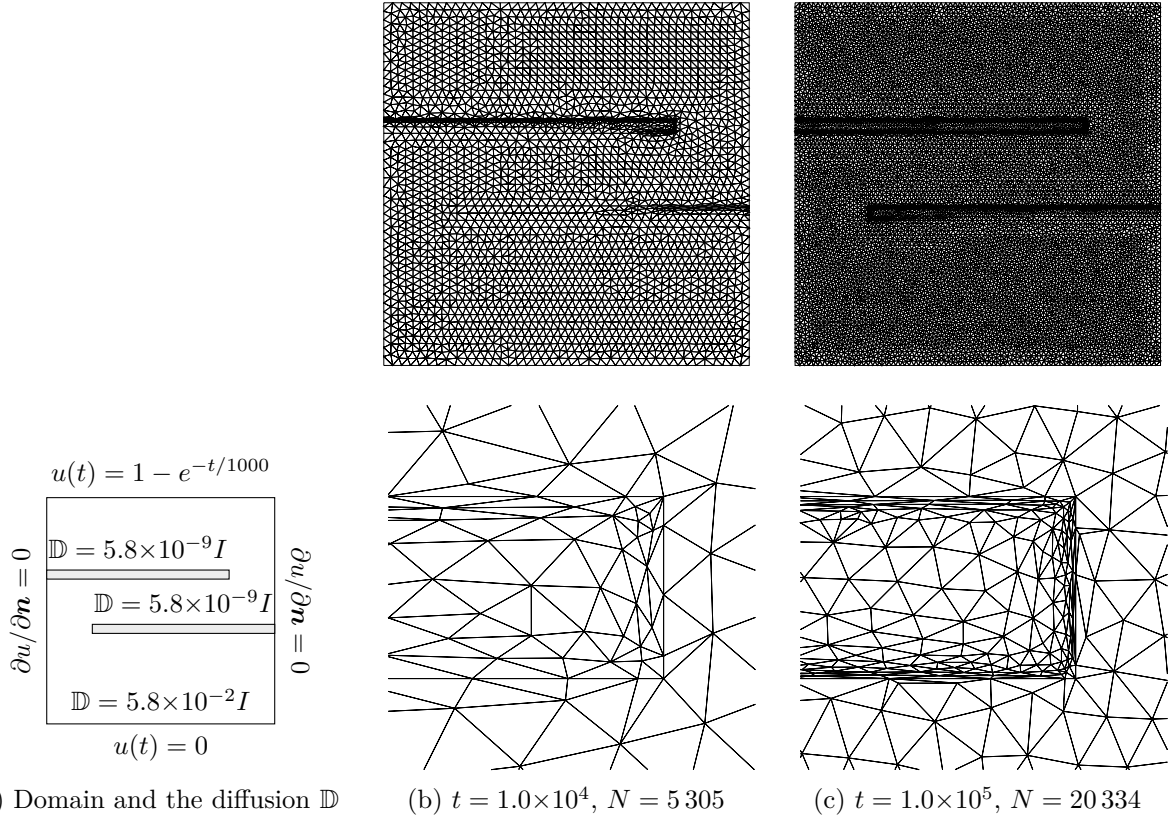
\begin{figure}[p]
   \subcaptionbox{Domain and the diffusion $\D$\label{fig:waterflow:domain}}
   [0.33\linewidth]
   {%
      \tikzsetnextfilename{hkl2013-waterflow-omega}%
      %--- waterflow domain ---
\begin{tikzpicture}[scale = 0.030]
   \tikzstyle{every node}=[font=\small]
   % nodes
   \path (   0,   0 ) coordinate (p01);
   \path ( 100,   0 ) coordinate (p02);
   \path ( 100,  40 ) coordinate (p03);
   \path (  20,  40 ) coordinate (p04);
   \path (  20,  44 ) coordinate (p05);
   \path ( 100,  44 ) coordinate (p06);
   \path ( 100, 100 ) coordinate (p07);
   \path (   0, 100 ) coordinate (p08);
   \path (   0,  68 ) coordinate (p09);
   \path (  80,  68 ) coordinate (p10);
   \path (  80,  64 ) coordinate (p11);
   \path (   0,  64 ) coordinate (p12);
   % Omega
   \draw [] (p01) -- (p02) -- (p07) -- (p08) --cycle;
   %\node [below] at ( 50, 20) {$\D = 1.16 \times 10^{-3} I$};
   \node [below] at ( 50, 20) {$\D = \num{5.8e-2} I$};
   % tunnel
   \draw [fill = gray!15] (p03) -- (p04) -- (p05) -- (p06) --cycle;
   \draw [fill = gray!15] (p09) -- (p10) -- (p11) -- (p12) --cycle;
   %\node [above] at ( 40, 68 ) {$\D = 1.16 \times 10^{-10} I$};
   %\node [above] at ( 60, 44 ) {$\D = 1.16 \times 10^{-10} I$};
   \node [above] at ( 40, 68 ) {$\D = \num{5.8e-9} I$};
   \node [above] at ( 60, 44 ) {$\D = \num{5.8e-9} I$};
   % boundary conditions
   \node [above]  at (  50, 100 ) {$u(t) = 1 - e^{-t / 1000}$};
   \node [below]  at (  50,   0 ) {$u(t) = 0$};
   \node [above, rotate=+90]   at (   0,  50 ) {${\partial u}/{\partial\bm{n}} = 0$};
   \node [above, rotate=-90]  at ( 100,  50 ) {${\partial u}/{\partial\bm{n}} = 0$};
\end{tikzpicture}%
   }%
   \hfill{}%
   \subcaptionbox{$t = \num{1.0e4}$, $N = \num{5305}$\label{fig:waterflow:iv}}
   [0.33\linewidth]
   {%
      \includegraphics[width=0.30\textwidth, clip]{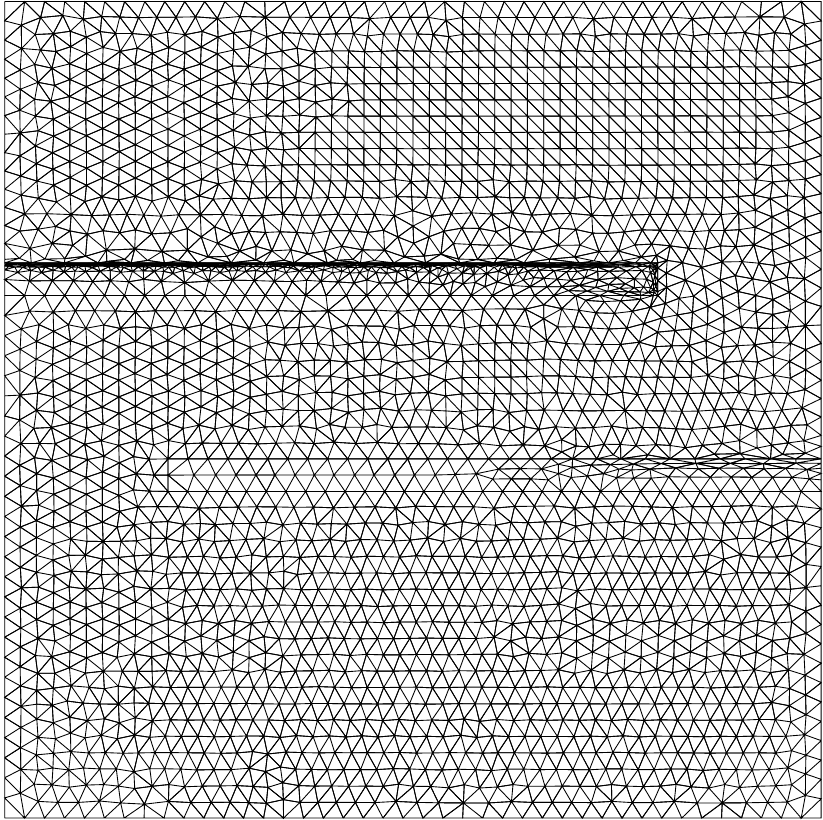}%
      \newline{}%
      \newline{}%
      \includegraphics[width=0.30\textwidth, clip]{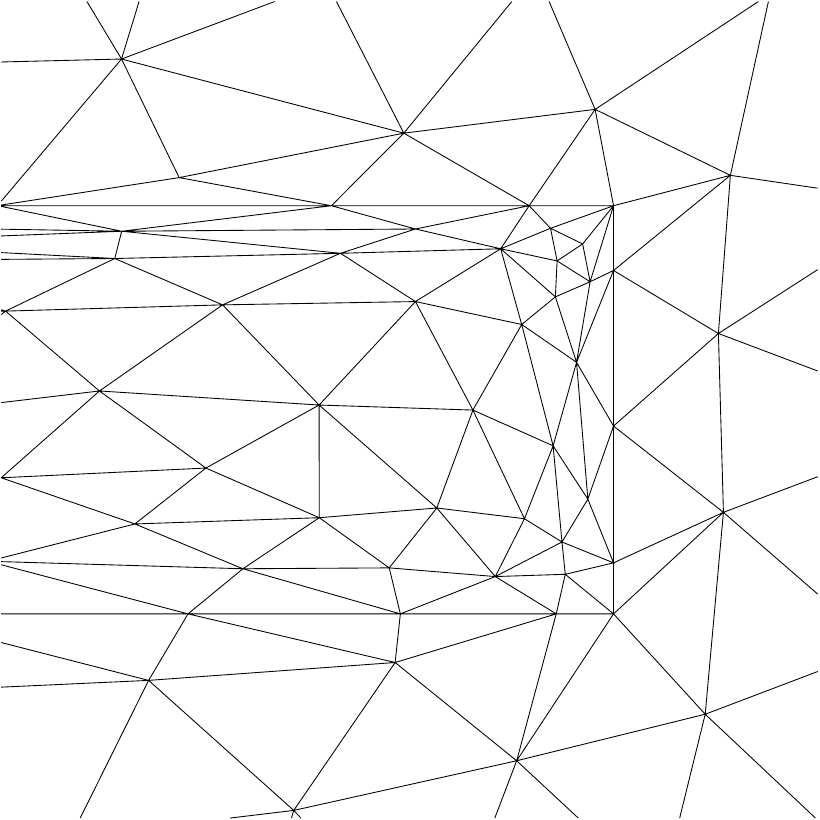}%
   }%
   \hfill{}%
   \subcaptionbox{$t = \num{1.0e5}$, $N = \num{20334}$\label{fig:waterflow:v}}
   [0.33\linewidth]
   {%
      \includegraphics[width=0.30\textwidth, clip]{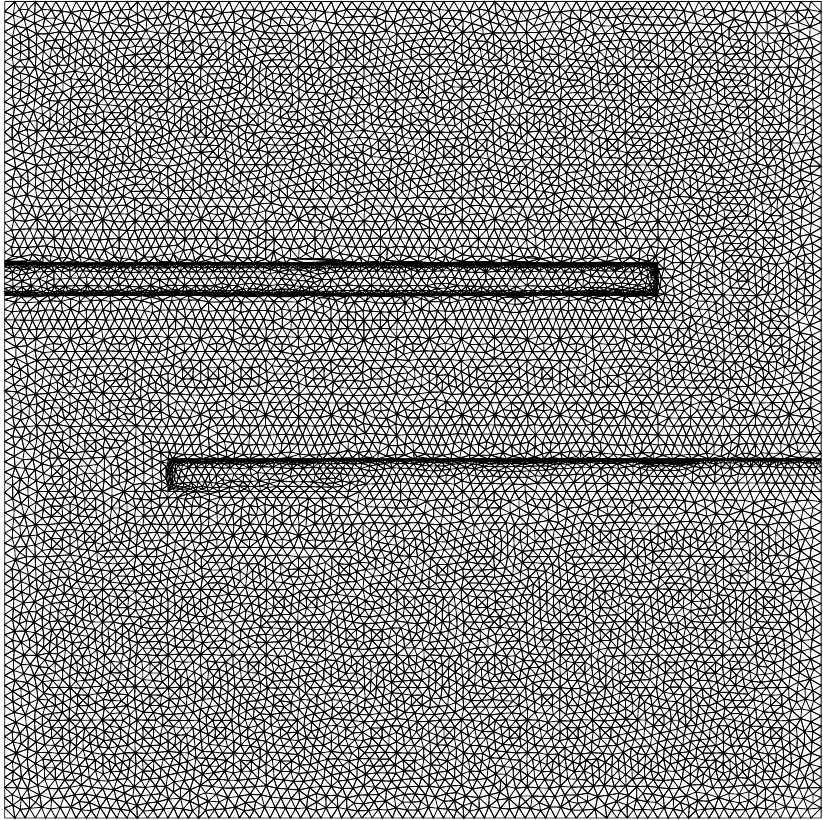}%
      \newline{}%
      \newline{}%
      \includegraphics[height=0.30\textwidth, clip]{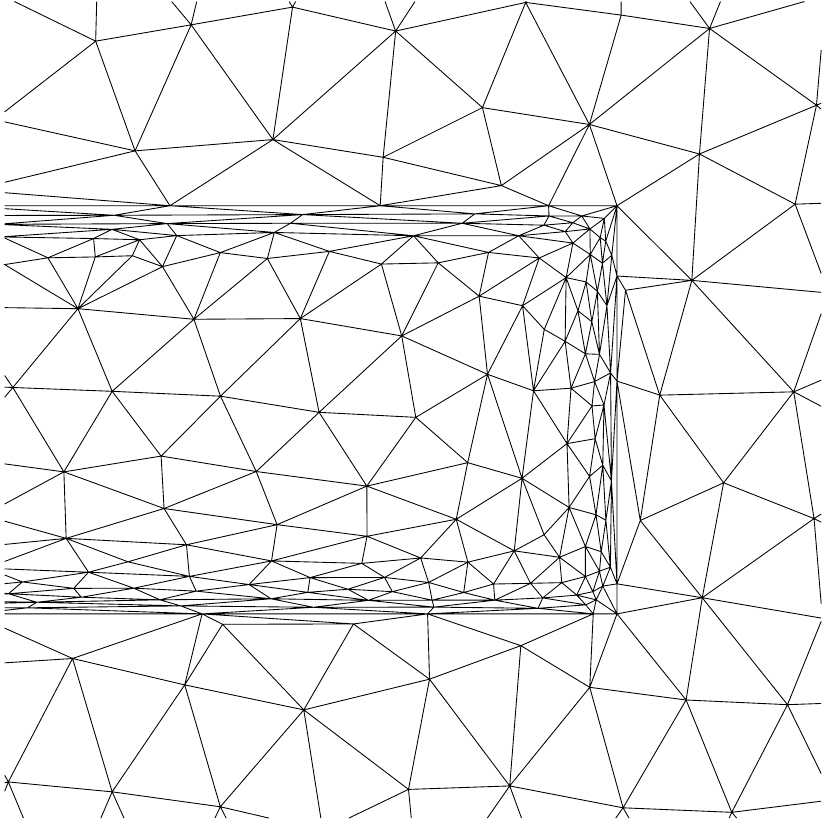}%
   }%
   \caption{%
      Domain, mesh examples and close-ups at $[74, 82] \times [62, 70]$
      (the upper right corner at the entrance of the tunnel)
      for the groundwater flow (\cref{ex:waterflow})\label{fig:waterflow}%
   }%
\end{figure}
\begin{table}[p]
   \caption{Numerical results
      for the groundwater flow (\cref{ex:waterflow})}
   %--- our new estimate -----------------------
   \begin{subtable}[c]{1.0\linewidth}
      \centering
      \subcaption{%
         computing $\tau_h$ with \cref{eq:tauh:mat}\label{tab:waterflow}%
      }
         \begin{tabular}[b]{@{} lr llr llr @{}}
            \toprule%
            && \multicolumn{3}{c}{with mass lumping}
            & \multicolumn{3}{c}{without mass lumping}\\
            \cmidrule(r){3-5}
            \cmidrule(r){6-8}
            \multicolumn{1}{@{}l}{\phantom{1}time}
            & $N\phantom{12}$
            & $\frac{\tau_{\max}}{s^2}$
            & $\frac{\tau_h}     {s^2}$
            & $\frac{\tau_{\max}}{\tau_h}$
            & $\frac{\tau_{\max}}{s^2}$
            & $\frac{\tau_h}     {s^2}$
            & $\frac{\tau_{\max}}{\tau_h}$ \\
            \midrule%
            \num{1.0e+02} &\num{  3071}  &\num{1.48e+00} &\num{5.97e-01} &2.48  &\num{5.77e-01} &\num{1.49e-01} &3.87 \\
\num{5.0e+03} &\num{  2799}  &\num{4.74e+00} &\num{2.23e+00} &2.13  &\num{1.81e+00} &\num{5.57e-01} &3.25 \\
\num{1.0e+04} &\num{  5305}  &\num{1.80e+00} &\num{8.01e-01} &2.25  &\num{6.89e-01} &\num{2.00e-01} &3.44 \\
\num{1.0e+05} &\num{ 20334}  &\num{2.05e-01} &\num{9.11e-02} &2.25  &\num{7.45e-02} &\num{2.28e-02} &3.27 \\

            \bottomrule%
      \end{tabular}
   \end{subtable}%
   \\[2ex]
   %--- five Lanczos steps -----------------------
   \begin{subtable}[c]{1.0\linewidth}
      \centering
      \subcaption{%
         computing $\tau_h$ with five steps of the Lánczos method
         using a random starting vector\label{tab:waterflow:lanczos}%
         }
         \begin{tabular}[b]{@{} lr llr llr @{}}
            \toprule%
            && \multicolumn{3}{c}{with mass lumping}
            & \multicolumn{3}{c}{without mass lumping}\\
            \cmidrule(r){3-5}
            \cmidrule(r){6-8}
            \multicolumn{1}{@{}l}{\phantom{1}time}
            & $N\phantom{12}$
            & $\frac{\tau_{\max}}{s^2}$
            & $\frac{\tau_h}     {s^2}$
            & $\frac{\tau_{\max}}{\tau_h}$
            & $\frac{\tau_{\max}}{s^2}$
            & $\frac{\tau_h}     {s^2}$
            & $\frac{\tau_{\max}}{\tau_h}$ \\
            \midrule%
            \num{1.0e+02} &\num{  3071}  &\num{1.48e+00} &\num{1.48e+00} &1.00  &\num{5.77e-01} &\num{5.64e-01} &1.02 \\
\num{5.0e+03} &\num{  2799}  &\num{4.74e+00} &\num{4.44e+00} &1.07  &\num{1.81e+00} &\num{1.77e+00} &1.02 \\
\num{1.0e+04} &\num{  5305}  &\num{1.80e+00} &\num{1.69e+00} &1.07  &\num{6.89e-01} &\num{6.46e-01} &1.07 \\
\num{1.0e+05} &\num{ 20334}  &\num{2.05e-01} &\num{1.98e-01} &1.03  &\num{7.45e-02} &\num{7.05e-02} &1.06 \\

            \bottomrule%
      \end{tabular}
   \end{subtable}%
\end{table}
%
%-------------------------------------------------------------------------------
%--- 2D, anisotropic diffusion
%-------------------------------------------------------------------------------
\begin{example}[2D anisotropic diffusion]
\label{ex:aniso}
\normalfont{%
This example shows the importance of the interplay between the major diffusion directions and the mesh geometry.

Consider the IBVP~\cref{eq:IBVP} in $\Omega = {\left(0,1\right)}^2 \backslash {\left[\frac{4}{9},\frac{5}{9}\right]}^2$ with the homogeneous Dirichlet boundary condition and
\[
   \D =
   \begin{bmatrix}
      \cos\theta   & - \sin\theta \\
      \sin \theta  &  \cos\theta
   \end{bmatrix}
   \begin{bmatrix}
     1000   & 0 \\
     0      & 1
   \end{bmatrix}
   \begin{bmatrix}
      \cos\theta     & \sin\theta \\
      - \sin\theta   & \cos\theta
   \end{bmatrix},
   \quad
   \theta = \pi \sin x \cos y.
\]

First, we consider quasi-uniform meshes (\cref{fig:aniso:quni}), for which elements are close to be uniform in shape and size, $F_K' \approx \Abs{K}^{\frac{1}{d}} I$ and $\norm{\FDF}_2 \approx \lambda_{\max}(\D) \norm{\FF}_2$.
Hence, using \cref{eq:stability:geo-2,eq:lmax:ZhuDu} provides comparable results, which is confirmed by the numerical results in \cref{tab:aniso}: for quasi-uniform grids \cref{eq:stability:geo-2} and \cref{eq:lmax:ZhuDu} or \cref{eq:lmax:Shewchuk} are accurate within a factor of \numrange{4.04}{6.35} and \numrange{4.52}{6.02}, respectively.

For  $\D^{-1}$-uniform (coefficient-adaptive) meshes  (\cref{fig:aniso:duni}) the situation is quite different and, as mentioned in \cref{rem:zhudu}, bound \cref{eq:stability:geo-2} should be more accurate than that using \cref{eq:lmax:ZhuDu}.
This is indeed confirmed by the numerical results: bound \cref{eq:stability:geo-2} is accurate within a factor of \numrange{3.40}{6.44}, whereas \cref{eq:lmax:ZhuDu} underestimates the real value by a factor of \numrange{347}{1020} (recalling $\kappa(\D) = \num{1000}$).
Note that Shewchuk's bound \cref{eq:lmax:Shewchuk} provides accurate results in any case, although not quite as accurate as \cref{eq:stability:geo-2}.
It is worth pointing out that the most accurate bound in all cases is \cref{eq:tauh:mat} in terms of the matrix entries.

This example also shows that $\D^{-1}$-uniform meshes allow larger time steps even if their elements may have ``bad quality'' in the common sense.
Hence, it is important to consider the quality of the mesh \emph{in relation to the diffusion} and not on itself.
}%
\end{example}

%--- Table for aniso -------------------------------------------------
\begin{figure}[p]
   \subcaptionbox{quasi-uniform\label{fig:aniso:quni}}
   [0.5\linewidth]
   {
      \tikzsetnextfilename{hkl2013-aniso-uni-mesh}
      \begin{tikzpicture}[scale=0.9]
         \begin{axis}[axis equal, hide axis,
               xmin = 0.0, xmax = 1.0, ymin = 0.0, ymax = 1.0,
               colormap={bw}{gray(0cm)=(0); gray(1cm)=(0)},
               height=0.40\linewidth,
            ]
            \addplot[patch, fill=white] table {hkl2013-aniso-uni-mesh.dat};
         \end{axis}
      \end{tikzpicture}
   }%
   \subcaptionbox{$\D^{-1}$-uniform\label{fig:aniso:duni}}
   [0.5\linewidth]
   {
      \tikzsetnextfilename{hkl2013-aniso-duni-mesh}
     \begin{tikzpicture}[scale=0.9]
         \begin{axis}[axis equal, hide axis,
               xmin = 0.0, xmax = 1.0, ymin = 0.0, ymax = 1.0,
               colormap={bw}{gray(0cm)=(0); gray(1cm)=(0)},
               height=0.40\linewidth,
         ]
            \addplot[patch, fill=white] table {hkl2013-aniso-duni-mesh.dat};
         \end{axis}
      \end{tikzpicture}
   }%
   \caption{Mesh examples for the anisotropic diffusion
      (\cref{ex:aniso})\label{fig:aniso}}
   \vspace{3ex}
   \captionof{table}{Numerical results for the anisotropic diffusion
      (\cref{ex:aniso})\label{tab:aniso}}
   %--- without mass lumping ------------------------
   \begin{subtable}[c]{1.0\linewidth}
   \centering
   \subcaption{%
      without mass lumping
   }%
      \begin{tabular}[b]{@{} rl ll ll lr @{}}
      \toprule%
       \multicolumn{2}{c}{ }
       & \multicolumn{2}{c}{new estimate \cref{eq:tauh:mat}}
       & \multicolumn{2}{c}{geometric \cref{eq:tauh:geo}}
       & \multicolumn{2}{c}{Zhu \& Du~\cite{ZhuDu14}}\\
      \cmidrule(r){3-4}
      \cmidrule(r){5-6}
      \cmidrule(r){7-8}
        $N\phantom{12}$
      & \phantom{12}$\frac{\tau_{\max}}{s^2}$
      & \phantom{12}$\frac{\tau_h}{s^2}$
      & $\frac{\tau_{\max}}{\tau_h}$
      & \phantom{12}$\frac{\tau_h}{s^2}$
      & $\frac{\tau_{\max}}{\tau_h}$
      & \phantom{12}$\frac{\tau_h}{s^2}$
      & $\frac{\tau_{\max}}{\tau_h}$ \\
      \midrule%
      \multicolumn{8}{c}{quasi-uniform meshes (\cref{fig:aniso:quni})} \\
      \midrule%
      \input{hkl2013-aniso-uni.dat}
      \midrule%
      \multicolumn{8}{c}{$\D^{-1}$-uniform meshes (\cref{fig:aniso:duni})} \\
      \midrule%
      \input{hkl2013-aniso-duni.dat}
      \bottomrule%
      \end{tabular}%
   \end{subtable}%
   \\[1ex]
   %--- with mass lumping ------------------------
   \begin{subtable}[c]{1.0\linewidth}
   \centering
   \subcaption{%
      with mass lumping
   }%
      \begin{tabular}[b]{@{} rl ll ll lr @{}}
      \toprule%
        \multicolumn{2}{c}{ }
       & \multicolumn{2}{c}{new estimate~\cref{eq:tauh:mat}}
        & \multicolumn{2}{c}{geometric~\cref{eq:tauh:geo}}
        & \multicolumn{2}{c}{Shewchuk~\cite{She02a}}\\
      \cmidrule(r){3-4}
      \cmidrule(r){5-6}
      \cmidrule(r){7-8}
        $N\phantom{12}$
      & \phantom{12}$\frac{\tau_{\max}}{s^2}$
      & \phantom{12}$\frac{\tau_h}{s^2}$
      & $\frac{\tau_{\max}}{\tau_h}$
      & \phantom{12}$\frac{\tau_h}{s^2}$
      & $\frac{\tau_{\max}}{\tau_h}$
      & \phantom{12}$\frac{\tau_h}{s^2}$
      & $\frac{\tau_{\max}}{\tau_h}$ \\
      \midrule%
      \multicolumn{8}{c}{quasi-uniform meshes (\cref{fig:aniso:quni})} \\
      \midrule%
      \input{hkl2013-aniso-uni-lump.dat}
      \midrule%
      \multicolumn{8}{c}{$\D^{-1}$-uniform meshes (\cref{fig:aniso:duni})} \\
      \midrule%
      \input{hkl2013-aniso-duni-lump.dat}
      \bottomrule%
   \end{tabular}%
   \end{subtable}%
\end{figure}

\newpage
%*** Conclusions **************************************************************
\section{Conclusions}
\label{sec:conclusion}

\Cref{thm:lmax:theorem} gives an easily computable bound on the largest eigenvalue of the system matrix $\tM^{-1}A$ in terms of the diagonal entries of $\tM$ and $A$ with $\tM$ being either $M$ or $M_{lump}$.
The bound is tight for \emph{any mesh} and \emph{any diffusion matrix $\D$} within a small constant which is given explicitly and depends only on the dimension of the domain.
This allows efficient and accurate estimation of the largest permissible time step $\tau_{\max}$.

Moreover, estimates \cref{eq:stability:geo,eq:stability:geo:ii} in terms of the mesh geometry reveals how the mesh and the diffusion matrix affect the stability condition.
Roughly speaking, $\tau_{\max}$ depends only on the number of mesh elements and the matching between the element geometry with the diffusion matrix.
Thus, it is not the element geometry itself but the \emph{element geometry in relation to the diffusion matrix} that is important for the stability.
The element quality measure $\QD$ provides a measure for the effect of a given element on the stability condition.
As seen in \cref{ex:ZhuDu}, strong anisotropic adaptation in the ``wrong'' direction can cause a significant reduction of the time step size.
Meanwhile, the result suggests that improvements in the element quality can significantly increase $\tau_{\max}$.

The achieved result can be extended for high order~\cite{HuaKamLan15} or even $p$-adaptive finite elements without major modifications.
Essentially, one only needs to recalculate the constants which depend on the choice of the basis functions.

Furthermore, numerical results suggest that, at least in one and two dimensions, mass lumping can increase the time step size by a factor of \numrange{2}{3}.
This topic deserves more detailed investigations.

%**********************************************************************
%*** Acknowledgement ******************************************************
%**********************************************************************
\section*{Acknowledgement}
Lennard Kamenski is thankful to Klaus Gärtner for a helpful comment that lead to \cref{rem:Duniform} and to Larissa Kaspar for providing parts of the code used in computations in \cref{ex:waterflow}.

The authors are grateful to an anonymous referee and particularly to Jed Brown for their valuable comments and suggestions which helped to improve the quality of this paper.

%\newpage
%*** Bibliography **************************************************************
%\bibliographystyle{siam}
%\bibliography{hkl2013}

%*** EOF ***********************************************************************
\end{document}